\documentclass[a4paper]{amsart}
\usepackage{multicol,ifthen}
\usepackage{amsmath}
\usepackage{amssymb}
\usepackage{amsthm} 
\usepackage{enumerate}

\begin{document}

\newcommand{\F}{\mathcal{F}}
\newcommand{\R}{\mathbb R}
\newcommand{\T}{\mathbb T}
\newcommand{\N}{\mathbb N}
\newcommand{\Z}{\mathbb Z}
\newcommand{\C}{\mathbb C}  
\newcommand{\h}[2]{\mbox{$ \widehat{H}^{#1}_{#2}(\R)$}}
\newcommand{\hh}[3]{\mbox{$ \widehat{H}^{#1}_{#2, #3}$}} 
\newcommand{\n}[2]{\mbox{$ \| #1\| _{ #2} $}} 
\newcommand{\x}{\mbox{$X^r_{s,b}$}} 
\newcommand{\xx}{\mbox{$X_{s,b}$}}
\newcommand{\X}[3]{\mbox{$X^{#1}_{#2,#3}$}} 
\newcommand{\XX}[2]{\mbox{$X_{#1,#2}$}}
\newcommand{\q}[2]{\mbox{$ {\| #1 \|}^2_{#2} $}}
\newcommand{\e}{\varepsilon}
\newcommand{\om}{\omega}
\newcommand{\lb}{\langle}
\newcommand{\rb}{\rangle}
\newcommand{\ls}{\lesssim}
\newcommand{\gs}{\gtrsim}
\newcommand{\pd}{\partial}
\newtheorem{lemma}{Lemma} 
\newtheorem{kor}{Corollary} 
\newtheorem{theorem}{Theorem}
\newtheorem{prop}{Proposition}

\title[mKdV and KdV hierarchies]{On the hierarchies of higher order mKdV and KdV equations}

\author[Axel~Gr{\"u}nrock]{Axel~Gr{\"u}nrock}

\address{Axel~Gr{\"u}nrock: Rheinische Friedrich-Wilhelms-Universit\"at Bonn,
Mathematisches Institut, Endenicher Allee 60, 53115 Bonn, Germany.}
\email{gruenroc@math.uni-bonn.de}

\thanks{The author was partially supported by the Deutsche Forschungsgemeinschaft, Sonderforschungsbereich 611.}

\subjclass[2000]{35Q53}

\begin{abstract}
The Cauchy problem for the higher order equations in the mKdV hierarchy is investigated with data in the spaces
$\widehat{H}^r_s(\R)$ defined by the norm
$$\n{v_0}{\widehat{H}^r_s(\R)} := \n{\langle \xi \rangle ^s\widehat{v_0}}{L^{r'}_{\xi}},\quad \langle \xi \rangle=(1+\xi^2)^{\frac12}, \quad \frac{1}{r}+\frac{1}{r'}=1.$$
Local well-posedness for the $j$th equation is shown in the parameter range $2 \ge r >1$, $s \ge \frac{2j-1}{2r'}$. The proof
uses an appropriate variant of the Fourier restriction norm method. A counterexample is discussed to show that the Cauchy problem
for equations of this type is in general ill-posed in the $C^0$-uniform sense, if $s<\frac{2j-1}{2r'}$. The results for $r=2$ - 
so far in the literature only if $j=1$ (mKdV) or $j=2$ - can be combined with the higher order conservation laws for the mKdV
equation to obtain global well-posedness of the $j$th equation in $H^s(\R)$ for $s\ge\frac{j+1}{2}$, if $j$ is odd, and for
$s\ge\frac{j}{2}$, if $j$ is even. - The Cauchy problem for the $j$th equation in the KdV hierarchy with data in
$\widehat{H}^r_s(\R)$ cannot be solved by Picard iteration, if $r> \frac{2j}{2j-1}$, independent of the size of $s\in \R$.
Especially for $j\ge 2$ we have $C^2$-ill-posedness in $H^s(\R)$. With similar arguments as used before in the mKdV context
it is shown that this problem is locally well-posed in $\widehat{H}^r_s(\R)$, if $1<r\le \frac{2j}{2j-1}$ and 
$s > j - \frac32 - \frac{1}{2j} +\frac{2j-1}{2r'}$. For KdV itself the lower bound on $s$ is pushed further down to
$s>\max{(-\frac12-\frac{1}{2r'},-\frac14-\frac{11}{8r'})}$, where $r\in (1,2)$. These results rely on the contraction mapping principle, and the flow map is real analytic.
\end{abstract}

\keywords{mKdV and KdV hierarchies -- Cauchy problem -- local and global well-posedness -- generalized Fourier restriction norm method}

\maketitle

\tableofcontents

\section{Introduction: Two towers}

The purpose of this paper is to investigate the Cauchy problem of the higher order
equations in the modified Korteweg-de Vries (mKdV) and Korteweg-de Vries (KdV) hierarchies. In preparation for that we recall
some basic facts about KdV, mKdV, the hierarchies of generalized equations built
upon them, and their connection by the Miura transform. For the KdV part we follow essentially the exposition of Lax \cite{L}\footnote{We replace consequently $u$ by $-6u$ in this exposition in order to achieve coincidence with reference \cite{VI}.}. The starting point is the well known sequence of polynomial densities
$$P_k = P_k(u, \pd_x u, \dots , \pd_x^k u),$$
starting with $P_{-1}=u$, $P_{0}=- \frac12 u^2$, $P_{1}=- \frac12 (\pd_xu)^2 + \frac16 u^3$, of the KdV equation
\begin{equation}
 \label{KdV}
\pd_tu + \pd_x^3u=6 u \pd_x u,
\end{equation}
which was discovered by Gardner et al. in \cite{II}, \cite{V}, \cite{VI}. This sequence is usually chosen in such a way that
\begin{itemize}
 \item the polynomials $P_k$ are irreducible, i.e. each term does not contain its highest derivative linearly, and
 \item the rank $r_{KdV}$ of all monomials contained in $P_k$ equals $k+2$ (For KdV the rank is defined by
 $r_{KdV}=m+\frac{n}{2}$, where $m$ is the number of factors (degree) and $n$ is the total number of differentiations 
(derivative index).)
\end{itemize}
See \cite[Sections 2 and 3]{V}. Under these assumptions the polynomials $P_k$
\begin{itemize}
 \item are uniquely determined up to a multiplicative constant \cite[Theorem 4]{V},
 \item contain $c (\pd^k_xu)^2$ as the highest derivative term (except for $k=-1$), where $c\neq 0 $ \cite[Theorem 6]{V}. 
\end{itemize}
The Hamiltonians
$$H_k(u):= \int P_k(u, \pd_x u, \dots , \pd_x^k u) dx$$
corresponding to the densities $P_k$ are constants of motion of \eqref{KdV} and lead to a sequence of a priori estimates for
the integer Sobolev norms $\|u\|_{H^k}$ of solutions of \eqref{KdV}. More precisely we have
\begin{equation}
 \label{apriori}
\|u\|^2_{H^k} \ls H_k(u) + f_k(\|u\|_{L^2}) \ls \|u\|^2_{H^k} + g_k(\|u\|_{L^2}) 
\end{equation}
with some positive nonlinear functions $f_k$ and $g_k$, see \cite[Lemme 4]{S79} and \cite[Theorem 3.1]{L75}. A recursion
formula for the gradients $G_k(u)$, defined by
\begin{equation}
 \label{grad}
\lb G_k(u),v\rb:= \lb\frac{\delta H_k(u)}{\delta u} ,v\rb:= \frac{d}{d \e}H_k(u+\e v) \Big{|}_{\e=0},
\end{equation}
 where $\lb \cdot, \cdot \rb$ denotes the inner product on $L^2$, goes back to Lenard, see \cite[Section 5]{VI}. We have
\begin{equation}
 \label{Lenard}
\pd_x G_{k+1}(u) = c_k N G_k(u),\quad \mbox{with} \quad N=N(u)=\pd_x^3 - 2(\pd_x u + u \pd_x).
\end{equation}
Here the operator $\pd_x u$ is to be understood as $\pd_x u (f)= (\pd_x u)f + u (\pd_x f)$. In 1968 Lax \cite{L68}
introduced the hierarchy of higher order KdV equations
\begin{equation}
 \label{hoKdV1}
\pd_tu = K_j(u)
\end{equation}
by the aid of the commutators $K_j(u)= [B_j(u),L]$ with the Schr\"odinger operator $L=-\pd_x^2+3u$ and the skew symmetric differential operator $B_j = \pd_x^{2j+1} + \sum_{k=0}^{j-1} b_{jk}\pd_x^{2k+1}+ \pd_x^{2k+1}b_{jk}$ with $b_{jk}$ suitably
chosen so that the commutators $[B_j(u),L]$ become multiplication operators. Following a discovery of Gardner he showed that
$K_j(u)=c_j \pd_x G_j(u)$ with the gradients defined above, so that \eqref{hoKdV1} becomes
\begin{equation}
 \label{hoKdV}
\pd_tu + \pd_xG_j(u)=0,
\end{equation}
which we take here as the \emph{definition} of the KdV hierarchy. Observe that \eqref{hoKdV} reduces to \eqref{KdV} in the
case when $j=1$. For $j=2$, choosing $c_1=-1$ in the recursion formula \eqref{Lenard}, we obtain
\begin{equation}
 \label{hoKdV2}
\pd_tu - \pd_x^5u + 5 \pd_x (\pd_x^2u^2 - (\pd_xu)^2-2u^3)=0
\end{equation}
and for $j=3$ we find ($c_1=c_2=-1$)
\begin{equation}
 \label{hoKdV3}
\pd_tu + \pd_x^7u - 7 \pd_x(\pd_x^4u^2 -2 \pd_x^2(\pd_xu)^2+ (\pd_x^2u)^2 - 10 u\pd_x(u\pd_xu)+5u^4)=0.
\end{equation}
To derive precise explicit expressions for the equations \eqref{hoKdV}  in general leads to hard and lengthy calculations
as were carried out in \cite{AS}. A less precise representation will be sufficient for our purposes. Since the operation
$\frac{\delta}{\delta u}$ reduces the number of factors by one, we get from the rank condition on the conserved densities
that there exists constants $c_{j,k,l}$, such that \eqref{hoKdV} can be rewritten as
\begin{equation}
 \label{hoexp1}
\pd_t u + c_{j,1,2j+1}\pd_x^{2j+1}u + \sum_{k=2}^{j+1}N_{jk}(u)=0,
\end{equation}
with $c_{j,1,2j+1} \neq 0$ and 
\begin{equation}
 \label{hoexp2}
N_{jk}(u)= \sum_{|l|=2(j-k)+3}c_{j,k,l}\pd_x^{l_0}\prod_{i=1}^k \pd_x^{l_i}u, \quad l_0 \ge 1.
\end{equation}
Clearly, this representation is not unique, and in general the equations \eqref{hoexp1}, \eqref{hoexp2} will not belong to
the KdV hierarchy. If considered on the real line, these equations admit rescaling. In fact a short calculation shows that,
if $u$ is a solution of \eqref{hoexp1}, \eqref{hoexp2} with index $j$, then so is $u_{\lambda}$, defined for $\lambda >0$ by
$u_{\lambda}(x,t)= \lambda ^2 u(\lambda x, \lambda^{2j+1}t)$. Especially we have $u_{\lambda}(x,0)= \lambda ^2 u(\lambda x,0)$,
which is independent of $j$. Consequently, the critical Sobolev regularity\footnote{determined by the demand that $\|u_{\lambda}(\cdot,0)\|_{\dot{H}^s(\R)}=\|u(\cdot,0)\|_{\dot{H}^s(\R)}$} of all these equations is $s = - \frac32$, just as for the
usual KdV equation.

\quad

According to a result of Gardner (cf. \cite[Section 2]{IV}, for a simpler argument see \cite[Theorem 3.2]{L76}) the Hamiltonians
$H_k$ are in involution with respect to both Poisson brackets $\{F,H\}_1:=\lb G_ F,\pd_x G_H\rb$ and $\{F,H\}_2:=\lb G_F,N G_H\rb$
($G_F$ being the gradient of F), that means we have
\begin{equation}
 \label{invol}
\{H_k,H_l\}_1=\lb G_k,\pd_x G_l\rb=0 \quad \mbox{and} \quad \{H_k,H_l\}_2=\lb G_k,N G_l\rb=0
\end{equation}
for all $k,l \ge -1$. Hence, if $u$ is a real valued solution of \eqref{hoKdV}, then
$$\frac{d}{dt}H_k(u)= \lb G_k(u),\pd_t u \rb = -\lb G_k(u),\pd_x G_j(u)\rb=0.$$
Thus the $H_k(u)$ are conserved quantities not only for the KdV equation itself but also for its higher order generalizations.
Consequently the a priori estimates \eqref{apriori} are equally valid for the solutions of \eqref{hoKdV}.

\quad

We turn to our main objective, the mKdV equation
\begin{equation}
 \label{mKdV}
\pd_tv + \pd_x^3v=6 v^2 \pd_x v
\end{equation}
and its higher order generalizations. Our primary references here are Olver \cite{O} and Adler-Moser \cite{AM}, see also \cite{CP},
\cite{ZC}, and \cite{Ma}. The key fact is the famous Miura transform \cite{I}
$$v \mapsto \pd_x v + v^2 =:u,$$
which maps a solution $v$ of \eqref{mKdV} to a solution $u$ of \eqref{KdV}. As was observed in \cite{II}, this immediately gives a sequence of conserved densities $\widetilde{P}_k(v):=P_{k-1}(\pd_xv+v^2)$, $k\ge 0$, of rank $k+1$, where for mKdV the rank
is defined by $r_{mKdV}=\frac12(m+n)$. (As above: $m=$degree, $n=$derivative index; for $k=-1$ the density $\widetilde{P}_{-1}(v)=v$ is immediate from the equation.) Now the corresponding Hamiltonians
$$\widetilde{H}_k(v) = \int \widetilde{P}_k(v)dx = H_{k-1}(\pd_xv+v^2)$$
are introduced as well as their gradients $\widetilde{G}_k(v)$, for which a short calculation shows the identity
\begin{equation}
 \label{mgrad}
\widetilde{G}_{k+1}(v)=(-\pd_x+2v)G_k(\pd_xv+v^2).
\end{equation}
The higher order mKdV equations can then be defined by
\begin{equation}
 \label{homKdV}
\pd_tv + \pd_x\widetilde{G}_j(v)=0
\end{equation}
in complete analogy with \eqref{hoKdV}\footnote{The construction reported on here has been carried out for a number of nonlinear
evolution equations, for which there exists a sequence of polynomial conserved densities of increasing order, see \cite[Chapter 5]{Ma} and \cite{O}. An interesting further example is the hierarchy of higher order sine-Gordon equations considered by Olver,
which looks like a potential version - $v$ replaced by $\pd_xv$ - of the mKdV hierarchy. See Example 6 in \cite{O}. Clearly the
arguments developed below to treat the Cauchy problem apply as well to this sequence of equations.}. For the Poisson bracket
$\{\cdot,\cdot\}_1$ one has by \eqref{mgrad}, \eqref{invol} and the easily checked identity
\begin{equation}
 \label{Gerd}
(\pd_x+2v)\pd_x(-\pd_x+2v)=-N(\pd_xv +v^2)
\end{equation}
that
$$\{\widetilde{H}_{k+1},\widetilde{H}_{l+1}\}_1=-\lb G_k(\pd_xv+v^2),N(\pd_xv+v^2) G_l(\pd_xv+v^2)\rb=0.$$
Hence we can conclude as in the case of the higher order KdV equations that $\frac{d}{dt}\widetilde{H}_k(v)=0$, if $v$ is a
real valued solution of \eqref{homKdV}. Again, the $\widetilde{H}_k(v)$ are conserved not only for solutions of mKdV itself
but also for those of \eqref{homKdV}, so that - after some applications of interpolation and Young's inequalities -  we can rely on the a priori estimates \eqref{apriori} with $u$ replaced by $v$
and $H_k$ replaced by $\widetilde{H}_k$ for all equations in the mKdV hierarchy. Moreover, combining \eqref{Lenard}, \eqref{Gerd},
and \eqref{mgrad} we obtain
$$\pd_x G_{j+1}(\pd_xv + v^2)= -c_j (\pd_x +2v)\pd_x \widetilde{G}_{j+1}(v),$$
so that with the choice $c_j=-1$ we have
$$\pd_xG_j(u)=(\pd_x+2v)\pd_x\widetilde{G}_j(v),$$
whenever $u$ is the Miura transform of $v$, i.e. $u=\pd_xv+v^2$. This gives
$$\pd_tu+\pd_xG_j(u)=(\pd_x+2v)(\pd_tv+\pd_x\widetilde{G}_j(v)).$$
Hence, if $v$ solves \eqref{homKdV} of index $j$, then its Miura transform $u$ solves the corresponding higher order KdV equation.

\quad

The recursion formula
$$\widetilde{G}_{k+1}(v)=\widetilde{c}_k(\pd_x^2 - 4v\pd_x^{-1}v\pd_x)\widetilde{G}_{k}(v)$$
with a formal antiderivative $\pd_x^{-1}$ is now easily derived using Lenard's formula \eqref{Lenard} and the identity
\eqref{Gerd}, alternatively we can write
\begin{equation}
 \label{Olver}
\pd_x \widetilde{G}_{k+1}(v)=\widetilde{c}_k(\pd_x^2-4v^2-4(\pd_xv)\pd_x^{-1}v)\pd_x\widetilde{G}_{k}(v),
\end{equation}
cf. (16) in \cite{O}. These recursion formulas can be used to derive exact explicit expressions for the higher order modified
equations, starting with mKdV itself, to which \eqref{homKdV} reduces in the case $j=1$. Choosing again $c_1=c_2=-1$ we obtain for
$j=2$
\begin{equation}
 \label{homKdV2}
\pd_tv - \pd_x^5v +  \pd_x (10(v^2\pd_x^2v +v (\pd_xv)^2)-6v^5)=0
\end{equation}
and for $j=3$
\begin{multline}
 \label{homKdV3}
\pd_tv + \pd_x^7v - 14 \pd_x ((5\pd_xv)^2\pd_x^2v+3v(\pd_x^2v)^2+4v(\pd_xv)(\pd_x^3v)+v^2(\pd_x^4v)) \\
+70 \pd_x (v^4\pd_x^2v +2 v^3 (\pd_xv)^2)-20\pd_x v^7=0,
\end{multline}
almost as in \cite[p. 151]{Ma}, where a different sign convention is used. The equation for $j=4$ can also be found in that reference - it takes four lines and contains 15 monomials. For $j \in \{1,2,3\}$ obviously $\pd_x \widetilde{G}_j$ is an odd
function of $v$, a property, which carries over inductively to the higher order equations by \eqref{Olver}. Combining this with
the rank condition for mKdV we find the following explicit expressions for \eqref{homKdV}:
\begin{equation}
 \label{homexp1}
\pd_t v + \widetilde{c}_{j,1,2j+1}\pd_x^{2j+1}v + \sum_{k=1}^{j}\widetilde{N}_{jk}(v)=0,
\end{equation}
where $\widetilde{c}_{j,1,2j+1} \neq 0$ and 
\begin{equation}
 \label{homexp2}
\widetilde{N}_{jk}(v)= \sum_{|l|=2(j-k)+1}\widetilde{c}_{j,k,l}\pd_x^{l_0}\prod_{i=1}^{2k+1} \pd_x^{l_i}v.
\end{equation}
Again these representations are not unique, and not for all choices of the constants $\widetilde{c}_{j,k,l}$ these equations
belong to the mKdV hierarchy. In the subsequent considerations we will always choose $c_{j,k,l}=\widetilde{c}_{j,k,l}=(-1)^{j+1}$.
Concerning the scale invariance the following is easily checked: If $v$ is a solution of \eqref{homexp1}, \eqref{homexp2}
defined on the real line, then so is $v_{\lambda}$ given by $v_{\lambda}(x,t)=\lambda v(\lambda x, \lambda ^{2j+1}t)$; the
critical Sobolev regularity for all equations in \eqref{homexp1}, \eqref{homexp2}, independent of $j$, is $s=-\frac12$,
as is well known in the case $j=1$ (usual mKdV).

\section{Statement of results}

\subsection{Concerning the higher order mKdV equations}

The Cauchy problem $v(x,0)=v_0(x)$ for the mKdV equation 
is known to be locally well posed for data $u_0$ in the classical Sobolev spaces $H^s (\R)$ if $s \ge \frac{1}{4}$, and ill posed in the sense that the mapping data upon solution is no longer uniformly continuous, if $s< \frac{1}{4}$. Both, the positive and the negative result, are due to Kenig, Ponce and Vega, see \cite[Theorem 2.4]{KPV93} and \cite[Theorem 1.3]{KPV01}, respectively.
So there is a considerable gap of $\frac34$ derivatives between the optimal local well-posedness result and the scaling
prediction. As was shown by the author in \cite{G04} and in a collaboration with Vega \cite{GV}, this gap can be closed except for the endpoint case, if data in the function spaces $\widehat{H}^r_s(\R)$ are considered, which are defined by the norm
$$\n{v_0}{\widehat{H}^r_s(\R)} := \n{\langle \xi \rangle ^s\F_x{v_0}}{L^{r'}_{\xi}},\quad \langle \xi \rangle=(1+\xi^2)^{\frac12}, \quad \frac{1}{r}+\frac{1}{r'}=1,$$
where $\F_x$ denotes the Fourier transform (in the space variable).
We remark that these spaces coincide with $B_{r',k}$ (with weight $k(\xi)=\lb\xi\rb^s$) 
introduced by H\"ormander, cf. \cite{Hoe83}, Section 10.1. The idea to consider them
as data spaces for nonlinear evolution equations goes back to the work of
Cazenave, Vega, and Vilela \cite{CVV01}, where corresponding weak $L^{r'}$-norms
are used. Yet another alternative class of data spaces has been considered by
Vargas and Vega in \cite{VV01}. In the more general setting of $\widehat{H}^r_s(\R)$-data the Cauchy problem for mKdV is locally well-posed
for $2\ge r >1$ and $s\ge \frac{1}{2r'}$, and ill-posed in the $C^0$-uniform sense mentioned above, if $s< \frac{1}{2r'}$, see
\cite[Theorem 1]{GV} and \cite[Section 5]{G04}. Similar results hold true for the cubic nonlinear Schr\"odinger equation in
one space dimension and the derivative nonlinear Schr\"odinger equation, see \cite{G05} for the case of the real line and \cite{GH} for the corresponding periodic problem. These well-posedness results were obtained by an appropriate variant of
Bourgain's Fourier restriction norm method, especially the function spaces $\x$ were used, which are given by their norms
$$\n{f}{X^r_{s,b}}:= \left(\int d \xi d \tau \langle \xi \rangle^{sr'}\langle \tau - \phi(\xi)\rangle^{br'} |\F{f}(\xi , \tau)|^{r'} \right) ^{\frac{1}{r'}},$$
where $\phi$ is the phase function associated with the linearized equation. For $r=2$ we write simply $X_{s,b}$ as usual.
The solution spaces in this argument are the time restriction norm spaces
$$X^r_{s,b}(\delta) := \{f = \tilde{f}|_{[-\delta,\delta] \times \R} : \tilde{f} \in X^r_{s,b}\}$$
endowed with the norm
$$\n{f}{X^r_{s,b}(\delta)}:= \inf \{ \n{\tilde{f}}{X^r_{s,b}} : \tilde{f}|_{[-\delta,\delta] \times \R} =f\}  .$$
The main result of the present paper is the following generalization of \cite[Theorem 1]{GV} to all the higher order equations
in the mKdV hierarchy. For its statement we fix $j$ and the phase function $\phi (\xi)=\xi^{2j+1}$ in the definition of $\x$
and $\x(\delta)$, respectively.
\begin{theorem}
 \label{mainres}
Let $2\ge r >1$, $s\ge s_j(r):=\frac{2j-1}{2r'}$ and $v_0 \in \widehat{H}^r_s(\R)$. Then there exist $\delta = \delta (\n{v_0}{\widehat{H}^{r}_{s(r)}(\R)}) > 0$ and a unique solution $v \in \x(\delta)$ of \eqref{homexp1}, \eqref{homexp2} with $v(x,0)=v_0(x)$. This solution is persistent
and the flow map $v_0 \mapsto v$, $\widehat{H}^r_s(\R) \rightarrow \x (\delta)$ is locally Lipschitz continuous.
\end{theorem}
Theorem \ref{mainres} is a consequence of the general local well-posedness Theorem from \cite[Theorem 2.3]{G04} and the
multilinear estimates in Section \ref{nonlin}, see Theorem \ref{cubic} and Theorem \ref{quintic+} below. The flow map
in our case is even real analytic, as follows by the implicit function Theorem. Theorem \ref{mainres} is valid not only for real, but also for complex valued functions $v$, and each factor in \eqref{homexp2} may be replaced by its complex conjugate (since
the phase functions are odd). This is important in view of the subsequent optimality result, where we consider the following
complex variant of equations \eqref{homexp1}, \eqref{homexp2}.
\begin{equation}
 \label{complex}
\pd_t v +(-1)^{j+1}\pd_x^{2j+1}v + \sum_{k=1}^j\sum_{l=0}^{2(j-k)+1}a_{jkl}(\pd_x^l|v|^{2k})(\pd_x^{2(j-k)+1-l}v)=0
\end{equation}
\begin{prop}
 \label{counter}
Let $r \in (1,2]$ and $j\ge 1$ be fixed. Then there exist coefficients $a_{jkl}$, such that the Cauchy problem for equation \eqref{complex} with data in $\widehat{H}^r_s(\R)$ is ill-posed, if $-\frac{1}{r'}<s<s_j(r)$, in the sense that the mapping $v_0 \mapsto v$
(data upon solution) from $\widehat{H}^r_s(\R)$ into any solution space $X_T$ continuously embedded in $C([0,T],\widehat{H}^r_s(\R))$ cannot be
uniformly continuous on bounded subsets of $\widehat{H}^r_s(\R)$.
\end{prop}
The proof of Proposition \ref{counter} will be carried out in Section \ref{contra}, where we adapt and generalize a counterexample
from \cite{KPV01}.

\quad

So the picture concerning the Cauchy problem for mKdV obtained in \cite{G04}, \cite{GV} is reproduced here for all the higher order equations in the hierarchy on a larger scale, increasing in $j$. For data in the classical Sobolev spaces $H^s(\R)$ the gap
between the critical regularity and the best possible local result, if uniformly continuous dependence of the solution on the data
is demanded, amounts to $\frac{2j+1}{4}$ derivatives. In each step from $j$ to $j+1$ we loose $\frac12$ derivative. Considering
data in the more general spaces $\widehat{H}^r_s(\R)$ we can close this gap almost completely. In view of both, the scaling argument and
the ill-posedness result in Proposition \ref{counter} the space $\h{1}{0}=:\widehat{L^1}(\R)$ becomes critical, and for all $j$ our positive result in Theorem \ref{mainres} gets arbitrary closed to it. Unfortunately, this critical space as well as its
various subspaces (finite measures, $\F_x^{-1}(C^0)$, $L^1$, \dots) are out of reach for our method of proof, even for small data,
and we must leave this as a challenging open question\footnote{In the case of the mKdV equation itself the existence of global
solutions for specific small data $v_0= \e_1 \delta + \e _2 p.v. \frac{1}{x}$ of critical regularity was obtained in \cite{PV},
related results for semilinear Schr\"odinger equations are in \cite{BV}.}.

\quad

In this paper the emphasis is on the improvement concerning local well-posedness achieved by considering the two parameter scale
$\widehat{H}^r_s(\R)$ of data spaces instead of restricting to the classical Sobolev spaces $H^s(\R)=\h{2}{s}$. But - to the author's knowledge - even the special case $r=2$ of Theorem \ref{mainres} is not yet in the literature, except for $j\in \{1,2\}$. The following results were previously known:
\begin{itemize}
 \item In 1979 Saut \cite{S79} proved global existence of persistent solutions of the $j$th equation in \eqref{homKdV} (and
 in \eqref{hoKdV}) with real valued data in $H^k$, $k\ge j$ integer. Using a priori estimates and parabolic regularization, his
 prove works as well for the corresponding periodic problem. The question of uniqueness was left open.
 \item In \cite{KPV94a} and \cite{KPV94b} Kenig, Ponce, and Vega showed, that there exist $s_0(j)$, such that the Cauchy problem
 for equation \eqref{homKdV} with index $j$ is locally well-posed in $H^s(\R)$ for all $s\ge s_0(j)$. Working in general in
 weighted spaces and on a larger class of polynomial nonlinearities, they explicitly point out in \cite[Theorem 1.2]{KPV94b} that
 no weights are needed, if only cubic and higher order terms appear. Their proof combines a gauge transform with smoothing
 estimates for the linearized equation.
\end{itemize}
Besides these two general results there are two papers, which are specifically concerned with the fifth order mKdV equation.
\begin{itemize}
 \item In \cite[Theorem 3.1]{L95} Linares showed that the Cauchy problem for \eqref{homKdV2} is globally well-posed in
 $H^2(\R \rightarrow \R)$\footnote{$H^s(\R \rightarrow \R)$ denotes the subspace of real valued functions in $H^s(\R)$.}. He also
 obtained a smoothing property of almost $\frac12$ derivative for the Duhamel term.
 \item An optimal local result for \eqref{homKdV2} was recently shown by Kwon \cite{K08b} using $X_{s,b}$-spaces and bilinear
 estimates. He obtained the $s\ge \frac34$-result and also showed ill-posedness in the $C^0$-uniform sense for lower
 regularities.
\end{itemize}
We observe that Kwon's local result, when combined with the conservation law at the level of $H^1$, gives global well-posedness
of the Cauchy problem for \eqref{homKdV2} in $H^s(\R\rightarrow \R)$, $s\ge 1$. For mKdV itself this was known before, see
Kenig, Ponce, and Vega \cite[p. 528]{KPV93}, and has been pushed down to lower regularities $s>\frac14$ by Colliander,
Keel, Staffilani, Takaoka, and Tao in \cite{CKSTT}. Combining the $r=2$-part of Theorem \ref{mainres} with the higher
conservation laws for the mKdV hierarchy, we obtain the following global result.

\begin{kor}
 \label{global}
Let $s\ge\frac{j+1}{2}$, if $j$ is odd, and $s\ge\frac{j}{2}$, if $j$ is even. Then the Cauchy problem for the higher order mKdV equation \eqref{homKdV} of index $j$ is globally well-posed in $H^s(\R \rightarrow \R)$.
\end{kor}
Furthermore our estimates imply, that the equations \eqref{homexp1}, \eqref{homexp2} are much better behaved, if no cubic
terms appear. In this case we can use Theorem \ref{quintic+Hs} below, which gives the following.
\begin{prop}
 \label{nocubic}
If in \eqref{homexp2} all the coefficients $\widetilde{c}_{j,1,l}$ vanish, then the Cauchy problem for \eqref{homexp1},
\eqref{homexp2} is locally well-posed in $H^s(\R)$ for $s>-\frac12$.
\end{prop}
We remark that global well-posedness in $H^s(\R)$, $s\ge0$, follows for the equations in Proposition \ref{nocubic}, if the
$L^2$-norm is preserved in time, e. g., if $v$ is real valued and $v\sum_{k=2}^j\widetilde{N}_{jk}(v)$ is an $x$-derivative.
A special example is
$$\pd_tv+(-1)^{j+1}\pd_x^{2j+1}v + v^{2j}\pd_xv=0.$$

\subsection{On the KdV hierarchy}

In view of the Miura transform we should expect the Cauchy problem for the higher order KdV equations \eqref{hoKdV} to be locally
well-posed for $s \ge s_j(2)-1=\frac{2j-5}{4}$. For $j=1$ this is indeed known to be true, cf. \cite{KPV96a}, for the endpoint see
\cite{CCT}. But for all the higher equations in the KdV hierarchy this heuristic is misleading. As was shown by Pilod in \cite{Pi08}, who used an argument of Molinet, Saut, and Tzvetkov \cite{MST} developed in the Benjamin-Ono context, the Cauchy
problem for the higher order KdV equations is ill-posed in any $H^s(\R)$-space in the sense that the flow map cannot be twice
continuously differentiable. (Strictly speaking, Pilod considers the special case of equations having only quadratic nonlinearities. But since the cubic and higher terms in \eqref{hoexp1} and \eqref{hoexp2} are well behaved, no cancellations occur, and his proof
applies as well to these more general equations.) This implies that for $H^s(\R)$-data - even for arbitrarily high regularities - no
local well-posedness result can be obtained by the contraction mapping principle. In this situation there are two alternatives. The
first is to lower the regularity assumptions on the flow map, so that, if merely continuous dependence of the solution on the data is demanded, energy methods can be applied successfully. For the fifth order equation \eqref{hoKdV2} this was carried out already
in 1993 by Ponce, see \cite{P93}. Compared with the result of Saut mentioned above, Ponce's argument gives also uniqueness and
continuous dependence. It is applicable for data in $H^s(\R)$, if $s>\frac72$, for $s\ge 4$ he obtains global well-posedness.
Recently, Kwon \cite{K08a} has used a refinement of the energy method due to Koch and Tzvetkov \cite{KT} to improve Ponce's result.
His lower threshold for local well-posedness is $s>\frac52$, and for $s \ge 3$ he gets a global result. The second alternative is
to leave the $H^s(\R)$-scale and to consider data in different function spaces, for example with weighted (in physical space) norms,
as it was carried out by Kenig, Ponce, and Vega in \cite{KPV94a}, \cite{KPV94b}, and in the sequel by Pilod \cite{Pi08}, who
considers small data in the intersection of $H^{2j+\frac14+\e}(\R)$ with a weighted Besov space. The latter results do rely on the
contraction mapping principle, thus yielding a smooth flow map defined near the origin of the data space.

\quad

Considering now data in the spaces $\h{r}{s}$, we first observe that the argument of Molinet-Saut-Tzvetkov and Pilod, respectively,
can be easily modified to show $C^2$-ill-posedness of the Cauchy problem for the $j$th equation in the KdV hierarchy for all
$s\in \R$, if $r>\frac{2j}{2j-1}$\footnote{In fact, if in the proof of \cite[Theorem 3]{Pi08} we completely replace 
$\|\quad\|_{H^s(\R)}$ by $\|\quad\|_{\widehat{H}^r_s(\R)}$, choose precisely the same $\alpha=N^{-2j-}$, and specialize to
$k=2j-1$, then the normalization condition (110) leads to an $\alpha^{-\frac{1}{r'}}$ instead of $\alpha^{-\frac{1}{2}}$ in
(107), (108), and hence to an $\alpha^{-\frac{2}{r'}}$ instead of $\frac{1}{\alpha}$ in the last line of (111). Consequently,
the lower bound on the right of (114) becomes $\frac{N^s}{N^s}\alpha^{-\frac{2}{r'}}N^{2j-1}\alpha^{\frac{1}{r'}}\alpha =
N^{2j-1}\alpha^{\frac{1}{r}}=N^{2j-1-\frac{2j}{r}-}$, which tends to infinity, if $r>\frac{2j}{2j-1}$, thus contradicting an
estimate - (106) in \cite{Pi08} - that would hold true in the case of $C^2$-regularity of the flow map.}. But for $r$ in the
small and shrinking interval $(1,\frac{2j}{2j-1}]$ it is in fact possible to obtain local well-posedness of the Cauchy problem
for \eqref{hoexp1}, \eqref{hoexp2} with data in $\h{r}{s}$, if $s$ is sufficiently large. To prove and quantify this, we shall
use the function spaces $X^{r,p}_{s,b}$ with norm
$$\|f\|_{X^{r,p}_{s,b}}:= \left(\int d \xi \left( d \tau \langle \xi \rangle^{sp'}\langle \tau - \phi(\xi)\rangle^{bp'} |\F{f}(\xi , \tau)|^{p'} \right)^{\frac{r'}{p'}} \right)^{\frac{1}{r'}},$$
where $\phi$ is the phase function of the linearized problem, in our case $\phi(\xi)=\xi^{2j+1}$ as above. The additional parameter
$p$ here will allow us to get a slightly better balance among the various cases in the subsequent estimates, which optimizes the
lower threshold for the Sobolev regularity. (Unfortunately, this beneficial effect becomes negligible for higher values of $j$,
since the allowed range for $p$ is strongly limited.) The time restriction norm spaces $X^{r,p}_{s,b}(\delta)$ are defined in the
same manner as for the $p=r$-variant used before.
\begin{theorem}
 \label{hoKdVlocal}
Let $j \ge 2$, $1< r \le p = \frac{2j}{2j-1}$, $s > j - \frac32 - \frac{1}{2j} +\frac{2j-1}{2r'}$, and $u_0 \in \widehat{H}^r_s(\R)$. Then there exist $\delta = \delta (\n{u_0}{\widehat{H}^{r}_{s}(\R)}) > 0$ and a unique solution $u \in X^{r,p}_{s,b}(\delta)$ of \eqref{hoexp1}, \eqref{hoexp2} with $u(x,0)=u_0(x)$. This solution is persistent
and the flow map $u_0 \mapsto u$, $\widehat{H}^r_s(\R) \rightarrow  X^{r,p}_{s,b}(\delta)$ is locally Lipschitz continuous.
\end{theorem}
Theorem \ref{hoKdVlocal} follows from the estimates in Section \ref{moreestimates} and the general theory in \cite[Section 2]{G04},
which can be easily adapted to the more general $X_{s,b}^{r,p}$-setting. The flow map here is again real analytic. For large $j$
the lower bound on $s$ becomes quite a desaster - in each step from $j$ to $j+1$ we loose more than a whole derivative, and not
only the scaling prediction but also the $s_j(2)-1=\frac{2j-5}{4}$ suggested by the Miura map comes out of reach rapidly. But
for $j=2$ we can allow regularities corresponding to $H^{-\frac14+\e}(\R)$, which are by far lower than the above mentioned
earlier results on this equation. So the question comes up naturally, if our arguments are sufficient to obtain an improvement
of the $H^{-\frac34}(\R)$-result for $j=1$, that is for KdV itself. The answer is yes, but in contrast to the mKdV equation we
stay away substantially from the critical regularity.

\begin{prop}
 \label{KdVlocal}
The Cauchy problem for the KdV equation with data $u_0 \in \widehat{H}^r_s(\R)$ is locally well-posed for $r\in (1,2)$ and
$s>\max{(-\frac12-\frac{1}{2r'},-\frac14-\frac{11}{8r'})}$. The local solutions belong to and are unique in a space 
$X_{s,b}^{r,p}(\delta)\subset C([0,\delta], \widehat{H}^r_s(\R))$ with phase function $\phi (\xi)=\xi^3$, $b=\frac{1}{p}+$ and
$$\frac{1}{p'} \quad \begin{cases} = \quad \frac14+\frac{5}{8r'}, \quad & \mbox{if} \quad 1<r \le \frac75 \\ \in \quad [\frac13(1+\frac{1}{r'}), \min{(\frac12,\frac{3}{2r'})}]
, \quad & \mbox{if} \quad \frac75 \le r \le 2\end{cases}
.$$
The positive lifespan $\delta$ depends on $\|u_0\|_{\widehat{H}^r_s(\R)}$, and the flow map is real analytic.
\end{prop}
The estimates necessary for Proposition \ref{KdVlocal} - as far as they are not already contained in the proof of Theorem
\ref{quadratic} - will be shown in the last section. We point out that in the whole interval $(1,2)$, where $r$ is admitted,
our lower bound on $s$ is strictly below the line $s=-\frac14-\frac{1}{r'}$, which corresponds via scaling to the $H^{-\frac34}(\R)$-result of \cite{CCT}.

\section{Estimates for free solutions of the generalized Airy equation}
 \label{Airy}

Throughout this section we consider solutions $u(t)=e^{(-1)^jt\partial_x^{2j+1}}u_0$, $v(t)=e^{(-1)^jt\partial_x^{2j+1}}v_0$,
and $w(t)=e^{(-1)^jt\partial_x^{2j+1}}w_0$ of the Airy type equation 
$$\partial_t u + (-1)^{j+1}\partial_x^{2j+1} u = 0, \quad j \ge 1,$$
with Cauchy data $u_0$, $v_0$ and $w_0$, respectively. Their Fourier transforms $\F _x u_0$, $\F _x v_0$, 
and $\F _x w_0$ are assumed to be nonnegative. Certain bi- and trilinear 
expressions involving these solutions will be estimated in the spaces $\widehat{L_x^p}(\widehat{L_t^q})$ 
and $\widehat{L^r_{xt}} := \widehat{L_x^r}(\widehat{L_t^r})$, where
$$\n{f}{\widehat{L_x^q}(\widehat{L_t^p})}:= \left( \int \Big{(} \int |\widehat{f}(\xi, \tau)|^{p'} d \tau \Big{)}^{\frac{q'}{p'}}d\xi \right)^{\frac{1}{q'}}, \,\,\,\,\,\frac{1}{q}+\frac{1}{q'}=\frac{1}{p}+\frac{1}{p'}=1.$$
A slight modification of the argument leading to the bilinear estimate will give us a linear estimate in the more
common spaces $L_t^p(L_x^q)$ with norm
$$\n{f}{L_t^p(L_x^q)}:= \left( \int \Big{(} \int |f(x, t)|^{q} d x \Big{)}^{\frac{p}{q}}dt \right)^{\frac{1}{p}}.$$

\subsection{Bilinear and linear estimates}

In order to state and prove the bilinear estimate, we introduce the bilinear pseudodifferential operator $M_{r,j}$,
which we define in terms of Fourier transforms by
$$\F_x M_{r,j}(f,g)(\xi)= \int_* d\xi_1 m_j(\xi_1,\xi_2)^{\frac{1}{r}}\F_x f(\xi_1)\F_x g(\xi_2),$$
where $\int_*$ is shorthand for $\int_{\xi_1+\xi_2=\xi}$ and the multiplier $m_j$ is given by
$m_j(\xi_1,\xi_2)=|\xi_1+\xi_2||\xi_1-\xi_2|(\xi_1^{2(j-1)}+\xi_2^{2(j-1)})$.

\begin{lemma}\label{Bill} 
 Let $1 \le q \le r_{1,2} \le p \le \infty$ and $\frac{1}{p}+\frac{1}{q}=\frac{1}{r_1}+\frac{1}{r_2}$. Then we have
\begin{equation}
 \label{Bronstein}
\|\F_xM_{p,j}(u,v)(\xi,\cdot)\|_{\widehat{L^p_t}}\ls(|\F_x{u_0}|^{p'}*|\F_x{v_0}|^{p'}(\xi))^{\frac{1}{p'}}
\end{equation}
and
\begin{equation}
 \label{Semendjajew}
\n{M_{p,j}(u,v)}{\widehat{L_x^q}(\widehat{L^p_t})}\ls \n{u_0}{\widehat{L^{r_1}_x}}\n{v_0}{\widehat{L^{r_{2}}_x}}.
\end{equation}
\end{lemma}

\begin{proof}
 Taking the Fourier transform first in space and then in time we obtain
$$ \F_x M_{p,j}(u,v)(\xi,t)= c \int_* d\xi_1m_{j}(\xi_1,\xi_2)^{\frac{1}{p}}e^{it(\xi_1^{2j+1}+\xi_2^{2j+1})}\F_xu_0(\xi_1)\F_xv_0(\xi_2)$$
and
$$ \F M_{p,j}(u,v)(\xi,\tau)= c \int_* d\xi_1m_j(\xi_1,\xi_2)^{\frac{1}{p}} \delta(\tau - \xi_1^{2j+1}- \xi_2^{2j+1})\F_xu_0(\xi_1)\F_xv_0(\xi_2),$$
respectively. Substituting $x=\xi_1-\frac{\xi}{2}$ we get
\begin{multline*}
 \F M_{p,j}(u,v)(\xi,\tau)= \\
c  \int  m_j(\tfrac{\xi}{2}+x,\tfrac{\xi}{2}-x)^{\frac{1}{p}} \delta((\tfrac{\xi}{2}+x)^{2j+1}+(\tfrac{\xi}{2}-x)^{2j+1}- \tau)\F_xu_0(\tfrac{\xi}{2}+x)\F_xv_0(\tfrac{\xi}{2}-x)dx
\end{multline*}
We use $\delta(g(x)) = \sum_n \frac{1}{|g'(x_n)|} \delta(x-x_n)$, where the sum is taken over all simple zeros of $g$, 
in our case
\begin{equation}\label{g1}
g(x)=(\tfrac{\xi}{2}+x)^{2j+1}+(\tfrac{\xi}{2}-x)^{2j+1}- \tau=\xi \sum_{l=0}^j\binom{ 2j+1 }{ 2l}(\tfrac{\xi}{2})^{2(j-l)}x^{2l}-\tau,
\end{equation}
which has only two zeros, say $\pm y$, where we take $y>0$. Then
\begin{equation}\label{g2}
 |g'(\pm y)|=|\xi| y\sum_{l=1}^j\binom{2j+1}{2l}2l(\tfrac{\xi}{2})^{2(j-l)}y^{2(l-1)} 
\gs m_j(\tfrac{\xi}{2}+y,\tfrac{\xi}{2}-y),
\end{equation}
and hence
\begin{multline}\label{201}
 \F M_{p,j}(u,v)(\xi,\tau) \\
 \ls m_j(\tfrac{\xi}{2}+y,\tfrac{\xi}{2}-y)^{-\tfrac{1}{p'}}\left(\F_xu_0(\tfrac{\xi }{2}+y)\F_xv_0(\tfrac{\xi }{2}-y) + \F_xu_0(\tfrac{\xi }{2}-y)\F_xv_0(\tfrac{\xi }{2}+y)\right). 
\end{multline}
Using $d\tau = g'(y)dy$, we see that the $L_{\tau}^{p'}$-norm of both contributions equals
$$\left(\int dy |\F_xu_0(\tfrac{\xi }{2}\pm y)\F_xv_0(\tfrac{\xi }{2}\mp y)|^{p'}\right)^{\frac{1}{p'}} = c \left(|\F_xu_0|^{p'}*|\F_xv_0|^{p'}(\xi) \right)^{\frac{1}{p'}},$$
which gives \eqref{Bronstein}. Now we choose $\rho'=\frac{q'}{p'}$  and $\rho_{1,2}$ with $\rho'_{1,2}= \frac{r'_{1,2}}{p'}$,
so that $\frac{1}{\rho}=\frac{1}{\rho_1}+\frac{1}{\rho_2}$. Then, using Young's inequality in the third step, we obtain
\begin{multline*}
 \n{M_{p,j}(u,v)}{\widehat{L_x^q}(\widehat{L^p_t})}\ls 
\left( \int d\xi (|\F_x{u_0}|^{p'}*|\F_x{v_0}|^{p'}(\xi))^{\frac{q'}{p'}}\right)^{\frac{1}{q'}} \\
= \||\F_x{u_0}|^{p'}*|\F_x{v_0}|^{p'}\|^{\frac{1}{p'}}_{L^{\rho'}_{\xi}}\ls 
\left(\n{|\F_x{u_0}|^{p'}}{L^{\rho'_1}_{\xi}}\n{|\F_x{v_0}|^{p'}}{L^{\rho'_2}_{\xi}}\right)^{\frac{1}{p'}}
=\n{u_0}{\widehat{L^{r_1}_x}}\n{v_0}{\widehat{L^{r_{2}}_x}},
\end{multline*}
which is the desired bound \eqref{Semendjajew}.
\end{proof}

\begin{kor}
 \label{xBill} 
For $p$, $q$, $r_{1,2}$ as in the previous lemma and $b > \frac{1}{p}$ the estimate
\begin{equation}
  \label{xBronstein}
\n{M_{p,j}(u_1, u_2)}{\widehat{L^q_x}(\widehat{L^p_t})} \ls \n{u_1}{X^{r_1,p}_{0,b}}\n{u_2}{X^{r_2,p}_{0,b}}
\end{equation}
is valid. Moreover, we have for $b_i > \frac{1}{r_i}$
\begin{equation}
 \label{xSemendjajew}
\n{M_{p,j}(u_1, u_2)}{\widehat{L^q_x}(\widehat{L^p_t})} \ls \n{u_1}{X^{r_1}_{0,b_1}}\n{u_2}{X^{r_2}_{0,b_2}}.
\end{equation}
\end{kor}

\begin{proof}
 We write $U(t)=e^{(-1)^jt\pd_x^{2j+1}}$ for the linear propagator and define, for $i \in \{1,2\}$, $g_i=\F_tU(-\cdot)u_i$,
where $\F_t$ denotes the Fourier transform in the time variable only. Then
$$u_i(t)= c \int e^{it \sigma_i}U(t) g_i(\sigma_i)d \sigma_i$$
and hence
$$M_{p,j}(u_1, u_2)(t)=c\int e^{it (\sigma_1 + \sigma_2)}M_{p,j}(U(t)g_1(\sigma_1),U(t)g_2(\sigma_2))d \sigma_1d \sigma_2.$$
Now we apply Minkowski's integral inequality, \eqref{Bronstein}, H\"older's inequality and Fubini's Theorem to obtain
\begin{eqnarray*}
 \n{\F_x M_{p,j}(u_1, u_2)(\xi)}{\widehat{L^p_t}} & \ls &
 \int \n{\F_x M_{p,j}(Ug_1(\sigma_1), Ug_2(\sigma_2))(\xi)}{\widehat{L^p_t}}d \sigma_1d \sigma_2 \\
& \ls & \int \left( |\F_xg_1(\sigma_1)|^{p'}*|\F_xg_2(\sigma_2)|^{p'}(\xi)\right)^{\frac{1}{p'}} d \sigma_1 d \sigma_2\\
& \ls & \left(\int \lb \sigma_1\rb^{bp'}\lb \sigma_2\rb^{bp'}|\F_xg_1(\sigma_1)|^{p'}*|\F_xg_2(\sigma_2)|^{p'}(\xi)d \sigma_1 d \sigma_2\right)^{\frac{1}{p'}} \\
& = & \left( \big( \int \lb \tau \rb^{bp'} |\F_xg_1(\tau)|^{p'}d \tau \big)* \big( \int \lb \tau \rb^{bp'} |\F_xg_1(\tau)|^{p'}d \tau \big) \right)^{\frac{1}{p'}}\\
& = & \left( \|\lb \tau \rb^b \F U(-\cdot)u_1\|^{p'}_{L^{p'}_{\tau}}*\|\lb \tau \rb^b \F U(-\cdot)u_2\|^{p'}_{L^{p'}_{\tau}}(\xi)\right)^{\frac{1}{p'}}.
\end{eqnarray*}
Now the $L^{q'}_{\xi}$-norm of this last expression is estimated by Young's inequality as at the end of the previous proof,
which gives \eqref{xBronstein}. Finally, two applications of H\"older's inequality lead to \eqref{xSemendjajew}.
\end{proof}
The next step is to dualize the estimate \eqref{xSemendjajew}. For that purpose we introduce the bilinear operator $M^*_{r,j}$,
which we define again in terms of Fourier transforms by
$$\F_x M^*_{r,j} (f,g) (\xi) := \int_*d\xi_1 m^*_j(\xi_1,\xi_2)^{\frac{1}{r}} \F_xf(\xi_1)\F_xg(\xi_2),$$
where
$$m^*_j(\xi_1,\xi_2)=|\xi_1||\xi_1+2\xi_2|((\xi_1+\xi_2)^{2(j-1)}+\xi_2^{2(j-1)}) \sim |\xi_1||\xi_1+2\xi_2|(\xi_1^{2(j-1)}+\xi_2^{2(j-1)}).$$
Then \eqref{xSemendjajew} in Corollary \ref{xBill} expresses the boundedness of
$$u_1 \mapsto M_{p,j}(u_1,u_2), \quad \X{r_1}{0}{b_1} \rightarrow \widehat{L^q_x}(\widehat{L^p_t})$$
with operator norm $\ls \n{u_2}{X^{r_2}_{0,b_2}}$.  By duality, under the additional hypothesis $1<p,q,r_{1,2}< \infty$, 
this implies the boundedness of
$$u_3 \mapsto M^*_{p,j}(u_3,\overline{u}_2), \quad   \widehat{L^{q'}_x}(\widehat{L^{p'}_t}) \rightarrow \X{r'_1}{0}{-b_1}$$
with the same norm. Using $\n{u_2}{X^{r}_{s,b}}=\n{\overline{u}_2}{X^{r}_{s,b}}$ we obtain the following estimate.
\begin{kor}\label{xBill*} Let $1 < q \le r_{1,2} \le p < \infty$, $\frac{1}{p}+\frac{1}{q}=\frac{1}{r_1}+\frac{1}{r_2}$ and 
$b_i > \frac{1}{r_i}$. Then
\begin{equation}\label{202}
\n{M^*_{p,j}(u_3,u_2)}{X^{r'_1}_{0,-b_1}}\ls \n{u_2}{X^{r_2}_{0,b_2}}\n{u_3}{\widehat{L^{q'}_x}(\widehat{L^{p'}_t})}. 
\end{equation}
\end{kor}
The special case in \eqref{202}, where $p=q=r_{1,2}$, will be sufficient for our purposes. In this case, \eqref{202} 
can be written as
\begin{equation}\label{203}
\n{M^*_{r',j}(u_3,u_2)}{X^{r}_{0,b'}}\ls \n{u_3}{\widehat{L^{r}_{xt}}}\n{u_2}{X^{r'}_{0,-b'}}. 
\end{equation}
provided $1 < r < \infty$, $b'<-\frac{1}{r'}$. Combining this with the trivial endpoint of the Hausdorff-Young 
inequality, i. e.
$$\n{u_3 u_2}{\widehat{L^{r}_{xt}}}\ls \n{u_3}{\widehat{L^{r}_{xt}}}\n{u_2}{\widehat{L^{\infty}_{xt}}},$$
we obtain by elementary H\"older estimates
\begin{equation}\label{204}
\n{M^*_{\rho',j}(u_3,u_2)}{X^{r}_{0,\beta}}\ls \n{u_3}{\widehat{L^{r}_{xt}}}\n{u_2}{X^{\rho'}_{0,-\beta}}, 
\end{equation}
where $0 \le \frac{1}{\rho'}\le \frac{1}{r'}$ and $\beta < -\frac{1}{\rho'}$. In this form actually we shall make use of Corollary \ref{xBill*}.

\quad

Combining the calculation in the proof of Lemma \ref{Bill} with Sobolev's embedding Theorem, Parseval's identity, and
the Hardy-Littlewood-Sobolev-inequality, we obtain the following linear estimate.

\begin{lemma}
\label{FeSt} For $4<q<\infty$ and $\frac{1}{r}=\frac{1}{2}+\frac{1}{q}$ the estimate
$$ \n{D_x^{\frac{2j-1}{4}}u}{L_t^4(L_x^q)} \ls \n{u_0}{\widehat{L^{r}_x}}$$
is valid.
\end{lemma}

\begin{proof}
 We assume first that $\F_x{u_0}= \chi_{[0,\infty)}\F_x{u_0}$ and write $v = D_x^{\frac{2j-1}{4}}u$. Then
$$ \|D_x^{\frac{2j-1}{4}}u\|^4_{L_t^4(L_x^q)} = \|v\|^4_{L_t^4(L_x^q)} 
= \||v|^2\|^2_{L_t^2(L_x^{\frac{q}{2}})} \ls \|D_x^{\e}|v|^2\|^2_{L_{xt}^2}=\|\F D_x^{\e}|v|^2\|^2_{L_{\xi \tau}^2}$$
for $\varepsilon = \frac{1}{2}-\frac{2}{q}$ by Sobolev's embedding theorem and Parseval's identity. Now with $x=\xi_1-\frac{\xi}{2}$ as in the proof of Lemma \ref{Bill} we have
$$\F D_x^{\e}|v|^2(\xi,\tau)= c|\xi|^{\e} \int\delta((\tfrac{\xi}{2}+x)^{2j+1}+(\tfrac{\xi}{2}-x)^{2j+1}- \tau)\F_xv_0(\tfrac{\xi}{2}+x)\overline{\F_xv_0}(x-\tfrac{\xi}{2})dx$$
with the same argument $g$ of $\delta$ as in \eqref{g1}, \eqref{g2}. By the support assumption on $\F_x u_0$ (and thus
$\F_x v_0$) we only get one contribution for the positive zero $y$ of $g$ and we have $2y = |\tfrac{\xi}{2}+y|+|\tfrac{\xi}{2}-y|$, so that $y$ controls both, the argument of $\F_x v_0$ as well as of $\overline{\F_x v_0}$.
This gives
\begin{multline*}
\F D_x^{\e}|v|^2(\xi,\tau) \ls  |\xi|^{\e - \frac12}\frac{1}{\sqrt{|g'(y)|}}\F_x (D_x^{-\frac{2j-1}{4}}v_0)(y+\tfrac{\xi}{2})
\overline{\F_x (D_x^{-\frac{2j-1}{4}}v_0)}(y-\tfrac{\xi}{2})\\
 =  |\xi|^{\e - \frac12}\frac{1}{\sqrt{|g'(y)|}}\F_xu_0(y+\tfrac{\xi}{2})\overline{\F_xu_0}(y-\tfrac{\xi}{2}). \hspace{3cm}
\end{multline*}
Squaring this last inequality and integrating with respect to $d \tau = g'(y) dy$ and to $d \xi$ we arrive at
\begin{multline*}
\|\F D_x^{\e}|v|^2\|^2_{L_{\xi \tau}^2}
\ls \int d \xi dy |\xi|^{2\e -1}|\F_xu_0(y+\tfrac{\xi}{2})\overline{\F_xu_0}(y-\tfrac{\xi}{2})|^2 \\
= \int d z_+ dz_- |z_+-z_-|^{2\e -1}|\F_xu_0(z_+)\overline{\F_xu_0}(z_-)|^2.\hspace{2.4cm}
\end{multline*}
Using the Hardy-Littlewood-Sobolev-inequality, which requires $0 < 1-2\e<1$ and $\frac{4}{r'}+1-2\e=2$, that is
$4<q<\infty$ and $\frac{1}{r}=\frac{1}{2}+\frac{1}{q}$ as assumed, we see that the latter is bounded by
$\||\F_x u_0|^2\|^2_{L^{\frac{r'}{2}}}= \|u_0\|^4_{\widehat{L}^r_x}$. So, in the special case where $\F_x{u}= \chi_{[0,\infty)}\F_x{u}$, the desired estimate is shown. Now, if $\F_x{w}= \chi_{(-\infty,0]}\F_x{w}$ and $u=\overline{w}$, then by $\F_x{\overline{w}}(\xi) = \overline{\F_x{w}}(-\xi)$ we see that $\F_x{u}= \chi_{[0,\infty)}\F_x{u}$. Thus the estimate is valid for $u$. Hence

$$\|D_x^{\frac{2j-1}{4}}w\|_{L_t^4(L_x^q)}  =  \|D_x^{\frac{2j-1}{4}}\overline{w}\|_{L_t^4(L_x^q)}
=\|D_x^{\frac{2j-1}{4}}u\|_{L_t^4(L_x^q)} \ls \n{u_0}{\widehat{L^{r}_x}} = \n{w_0}{\widehat{L^{r}_x}}.$$
Finally the decomposition $u=u_+ + u_-$ with $\F_x{u_+}= \chi_{[0,\infty)}\F_x{u}$ yields the desired result in the general case.
\end{proof}

The endpoint case $(p,q)=(4,\infty)$ is known to be true, too, see Theorem 2.1 in \cite{KPV91}. Next we interpolate among
Lemma \ref{Bill}, the conservation of the $L^2$ -  norm and the trivial estimate

$$\n{u}{L^{\infty}_{xt}} \ls \n{u_0}{\widehat{L^{\infty}_x}}$$
to obtain the following generalization.

\begin{kor}
 \label{genFeSt} Let $\frac{1}{r} = \frac{2}{p} + \frac{1}{q}$. Then the estimate

$$ \n{D_x^{\frac{2j-1}{p}}u}{L_t^p(L_x^q)} \ls \n{u_0}{\widehat{L^{r}_x}}$$

holds true, if one of the following conditions is fulfilled:
\begin{itemize}
\item[i)] $0 \le \frac{1}{p} \le \frac{1}{4}$, $0 \le \frac{1}{q} < \frac{1}{4}$ or
\item[ii)] $\frac{1}{4} \le \frac{1}{q} \le \frac{1}{q} + \frac{1}{p} < \frac{1}{2}$ or
\item[iii)] $(p,q) = (\infty , 2)$.
\end{itemize}
If in addition $b> \frac{1}{r}$, then
\begin{equation}
 \label{xFeSt}
\|D_x^{\frac{2j-1}{p}}u_1\|_{L_t^p(L_x^q)} \ls \|u_1\|_{X^r_{0,b}}.
\end{equation}

\end{kor}

The case $p=q$ in the preceeding Lemma is of special interest, here the conditions reduce to $0 \le \frac{1}{p} = \frac{1}{3r} < \frac{1}{4}$. The corresponding estimate for the Schr\"odinger equation goes back to Fefferman and Stein \cite{F70}.
Unfortunately these linear estimates fail to be true for $r \le \frac{4}{3}$ and are thus not sufficient for our purposes;
in our context they will mainly be used for several interpolation arguments.\footnote{Nonetheless it is the author's belief
that these estimates are of independent interest. For example, if combined with duality and the Christ-Kiselev-Lemma (see \cite{CK})
they imply a wider range of Strichartz type estimates for the solutions of the corresponding inhomogeneous equations than
previously known. More precisely, if $(p,q)$ and $(\tilde{p},\tilde{q})$ are two pairs of H\"older exponents admitted in Corollary
\ref{genFeSt}, which in addition satisfy $\frac{2}{p}+\frac{1}{q}+\frac{2}{\tilde{p}}+\frac{1}{\tilde{q}}=1$, then
$$\|D_x^{(2j-1)(\frac{1}{p}+\frac{1}{\tilde{p}})}\int_0^t e^{-(t-s)\pd_x^{2j+1}}F(s)ds\|_{L_t^p(L_x^q)} \ls \|F\|_{L_t^{\tilde{p}'}(L_x^{\tilde{q}'})}.$$ In view of the recent work of Vilela \cite{MCV}, Foschi \cite{Fo}, Taggart \cite{T},
and Ovcharov \cite{Ov}, the assumption $p, \tilde{p}\ge4$ seems to be redundant for this last estimate.} To overcome this difficulty we will prove in the sequel certain trilinear
estimates for free solutions, where - at least in one of two cases - the singularity can be distributed on two factors, so that
two applications of the HLS-inequality allow the full range $r>1$.

\subsection{Trilinear estimates}

To prove the trilinear estimates we calculate the Fourier transform in space and time of the product of three
free solutions. Similarly as in the proof of Lemma \ref{Bill} we obtain

\begin{multline*}
 \F (uvw)(\xi,\tau)= c \int_* d\xi_1 d\xi_2\delta(\xi_1^{2j+1}+ \xi_2^{2j+1}+ \xi_3^{2j+1}-\tau)
 \F_xu_0(\xi_1)\F_xv_0(\xi_2)\F_xw_0(\xi_3)\\
=c \int d\xi_1 dx \delta(g(\xi_1;x))\F_xu_0(\xi_1)\F_xv_0(\tfrac{\xi - \xi_1}{2}+x)\F_xw_0(\tfrac{\xi - \xi_1}{2}-x),
\end{multline*}

where now $\int_* = \int_{\xi_1+\xi_2+\xi_3=\xi}$, $x=\xi_2-\frac{\xi - \xi_1}{2}$, and

\begin{equation}
 \label{g3}
g(\xi_1; x)=\xi_1^{2j+1} - \tau +(\xi -\xi_1) \sum_{l=0}^j\binom{ 2j+1 }{ 2l}(\tfrac{\xi - \xi_1}{2})^{2(j-l)}x^{2l},
\end{equation}
 which has, for $\xi_1$ fixed, only two zeros, say again $\pm y$, with $y>0$. We have
\begin{multline}
 \label{g4}
|g'(\xi_1; \pm y)|:=|\frac{\partial g}{\partial x}(\xi_1; \pm y)|= 
|\xi - \xi_1| y\sum_{l=1}^j\binom{2j+1}{2l}2l(\tfrac{\xi - \xi_1}{2})^{2(j-l)}y^{2(l-1)} \\
\gs |\xi - \xi_1| y (|\xi - \xi_1|^{2(j-1)}+y^{2(j-1)}) \hspace{4cm}
\end{multline}

and hence $\F (uvw)(\xi, \tau) \ls K_+(\xi, \tau)+K_-(\xi, \tau)$, where

$$K_{\pm}(\xi, \tau)= \int d\xi_1 \frac{1}{|g'(\xi_1; \pm y) |}\F_xu_0(\xi_1)\F_xv_0(\tfrac{\xi-\xi_1}{2}\pm y)\F_xw_0(\tfrac{\xi-\xi_1}{2}\mp y)$$
with $|g'(\xi_1; \pm y)|$ as in \eqref{g4}. After these preparations we turn to estimate 
$\n{uvw}{\widehat{L^r_{xt}}}$ in the following cases:
\begin{itemize}
 \item[i)] $|\xi_2-\xi_3|\ge|\xi_2+\xi_3|$, that is $2y \ge |\xi - \xi_1|$,
 \item[ii)] $1\le|\xi_2-\xi_3|\le|\xi_2+\xi_3|$, i. e. $1 \le 2y \le |\xi - \xi_1|$.
\end{itemize}
For that purpose we introduce the trilinear operators $T_{\ge}$ and $T_{\le}$, which we define by
$$\F_x T_{\ge}(f,g,h):= \int_* d\xi_1  d\xi_2 \F_xf(\xi_1)\F_xg(\xi_2)\F_xh(\xi_3)\chi_{\{|\xi_2-\xi_3| \ge |\xi_2+\xi_3|\}},$$
and
$$\F_x T_{\le}(f,g,h):= \int_* d\xi_1  d\xi_2 \F_xf(\xi_1)\F_xg(\xi_2)\F_xh(\xi_3)\chi_{\{1\le|\xi_2-\xi_3| \le |\xi_2+\xi_3|\}}.$$
For $T_{\ge}(u,v,w)$ we have the following estimate.

\begin{lemma}
\label{Trill1}
Let $1<p_1<p<p_0<\infty$, $p<p'_0$, $\frac{3}{p}=\frac{1}{p_0}+\frac{2}{p_1}$ and $\frac{2}{p_1}<1+\frac{1}{p}$. Then the estimate
$$\n{T_{\ge}(u,v,w)}{\widehat{L^p_{xt}}} \ls \n{u_0}{\widehat{L^{p_0}_{x}}}\n{D_x^{-\frac{2j-1}{2p}}v_0}{\widehat{L^{p_1}_{x}}}\n{D_x^{-\frac{2j-1}{2p}}w_0}{\widehat{L^{p_1}_{x}}}$$
is valid.
\end{lemma}

\begin{proof}
For the Fourier transform of $T_{\ge}(u,v,w)$ in both variables we obtain
$$\F T_{\ge}(u,v,w)(\xi,\tau)= c (K^+_{\ge}(\xi,\tau) + K^-_{\ge}(\xi,\tau)),$$
where
$$ K^{\pm}_{\ge}(\xi,\tau)=\int_{\{2y\ge|\xi-\xi_1|\}}  \frac{d\xi_1}{|g'(\xi_1;y)|}\F_xu_0(\xi_1)\F_xv_0(\tfrac{\xi-\xi_1}{2}\pm y)\F_xw_0(\tfrac{\xi-\xi_1}{2}\mp y).$$
By symmetry we may restrict ourselves to the estimation of $K^+_{\ge}$. Using \\ $|\frac{\xi-\xi_1}{2}\pm y|\le 2y$ and H\"older's
inequality, we see that
\begin{multline*}
K^+_{\ge}(\xi,\tau) \ls \left( \int d \xi_1 \frac{|\F_x u_0 (\xi_1)|^p}{|\xi-\xi_1|^{(1-\theta)p}}\right)^{\frac{1}{p}} \times \\
   \left( \int \frac{d \xi_1}{|\xi-\xi_1|^{\theta p'-1}|g'(\xi_1;y)|}|\F_xD_x^{-\frac{2j-1}{2p}}v_0(\tfrac{\xi-\xi_1}{2}+ y)\F_xD_x^{-\frac{2j-1}{2p}}w_0(\tfrac{\xi-\xi_1}{2}- y)|^{p'}\right)^{\frac{1}{p'}},
\end{multline*}
where $\theta = \frac{3}{p'}-\frac{2}{p_1'}$ ($\in (0,1)$ by our assumptions). Taking the $L^{p'}_{\tau}$-norm 
of both sides and using $d\tau =  g'(\xi_1;y)dy$ we arrive at
\begin{multline*}
 \n{\F T_{\ge}(u,v,w)(\xi,\cdot)}{L^{p'}_{\tau}} \ls (|\F_x u_0 |^p * |\xi|^{(\theta - 1)p})^{\frac{1}{p}}\times \\
 \left( \int \frac{d \xi_1 dy}{|\xi-\xi_1|^{\theta p' - 1}}|\F_xD_x^{-\frac{2j-1}{2p}}v_0(\tfrac{\xi-\xi_1}{2}+ y)\F_xD_x^{-\frac{2j-1}{2p}}w_0(\tfrac{\xi-\xi_1}{2}- y)|^{p'}\right)^{\frac{1}{p'}}.
\end{multline*}
Changing variables ($z_{\pm}:=\frac{\xi-\xi_1}{2}\pm y$) we see that the second factor equals
\begin{multline*}
 \left(\int \frac{d z_+ dz_-}{|z_++z_-|^{\theta p' - 1}}
 |\F_xD_x^{-\frac{2j-1}{2p}}v_0(z_+)\F_xD_x^{-\frac{2j-1}{2p}}w_0(z_-)|^{p'}\right)^{\frac{1}{p'}} \\
\ls \n{D_x^{-\frac{2j-1}{2p}}v_0}{\widehat{L^{p_1}_{x}}} \n{D_x^{-\frac{2j-1}{2p}}w_0}{\widehat{L^{p_1}_{x}}}, \hspace{3cm}
\end{multline*}
by the Hardy-Littlewood-Sobolev-inequality, requiring $\theta$ to be chosen as above and $1<\theta p'<2$, which follows from our assumptions.
It remains to estimate the $L^{p'}_{\xi}$-norm of the first factor, that is
$$\||\F_x{u_0}|^p* |\xi|^{(\theta -1)p}\|^{\frac{1}{p}}_{L_{\xi}^{\frac{p'}{p}}}
\ls ( \n{|\F_x{u_0}|^p}{L_{\xi}^{\frac{p'_0}{p}}} \n{|\xi|^{(\theta -1)p}}{L_{\xi}^{\frac{1}{(1-\theta)p}, \infty}} )^{\frac{1}{p}}
\ls \n{u_0}{\widehat{L^{p_0}_x}},$$
where the HLS inequality was used again. For its application we need
$$0<(1-\theta)p<1;\,\,\,\,1<\frac{p'_0}{p}< \frac{1}{1-(1-\theta)p}\,\,\,\,\mbox{and}\,\,\,\,\,\theta=\frac{1}{p'_0},$$
which follows from the assumptions, too.

\end{proof}

\begin{kor}
\label{Trill}
For $1<r \le 2$ there exist $s_{0,1}\ge 0$ with $s_0 + 2s_1 = \frac{2j-1}{r}$, such that
\begin{equation}\label{209}
\n{T_{\ge}(u,v,w)}{\widehat{L^r_{xt}}} \ls \n{D_x^{-s_0}u_0}{\widehat{L^{r}_{x}}}\n{D_x^{-s_1}v_0}{\widehat{L^{r}_{x}}}\n{D_x^{-s_1}w_0}{\widehat{L^{r}_{x}}}.
\end{equation}
In addition, for $b>\frac{1}{r}$ we have
\begin{equation}
\label{xTrill}
\n{T_{\ge}(u_1,u_2,u_3)}{\widehat{L^r_{xt}}} \ls \n{D_x^{-s_0}u_1}{X^{r}_{0,b}}\n{D_x^{-s_1}u_2}{X^{r}_{0,b}}\n{D_x^{-s_1}u_3}{X^{r}_{0,b}}.
\end{equation}
\end{kor}

\begin{proof}[Proof of \eqref{209}]
  Using H\"older's inequality and Corollary \ref{genFeSt}, that is
\begin{equation}\label{FS}
\n{u}{L^{3q}_{xt}} \ls \n{D_x^{-\frac{2j-1}{3q}}u_0}{\widehat{L^q_{x}}},\hspace{1cm}q>\frac{4}{3},
\end{equation}
we get for
\begin{equation}\label{210}
\frac{4}{3}< q_0<2<q_1,\hspace{1cm}\mbox{with}\hspace{1cm}\frac{3}{2}=\frac{1}{q_0}+\frac{2}{q_1}
\end{equation}
that
\begin{equation}\label{211}
\n{T_{\ge}(u,v,w)}{L^{2}_{xt}} \le \n{uvw}{L^{2}_{xt}} \ls \n{D_x^{-\frac{2j-1}{3q_0}}u_0}{\widehat{L^{q_0}_{x}}}\n{D_x^{-\frac{2j-1}{3q_1}}v_0}{\widehat{L^{q_1}_{x}}}\n{D_x^{-\frac{2j-1}{3q_1}}w_0}{\widehat{L^{q_1}_{x}}}.
\end{equation}
Multilinear interpolation of (\ref{211}) with Lemma \ref{Trill1} yields (\ref{209}), provided $p,p_0,p_1$; $q_0,q_1$, defined by the interpolation conditions
$$\frac{1}{r}=\frac{1- \theta}{p}+\frac{\theta}{2}=\frac{1-\theta}{p_0}+\frac{\theta}{q_0}=\frac{1-\theta}{p_1}+\frac{\theta}{q_1},$$
fulfill the assumptions of Lemma \ref{Trill1} and (\ref{210}), respectively, which can be guaranteed by choosing $\theta$ 
sufficiently small. Now $s_{0,1}$ are obtained from
$$s_0=\frac{(2j-1)\theta}{3q_0}\hspace{1cm}\mbox{and}\hspace{1cm}s_1=\frac{(2j-1)(1-\theta)}{2p}+\frac{(2j-1)\theta}{3q_1},$$
which gives
$$s_0 + 2s_1 =\frac{(2j-1)(1-\theta)}{p}+ \frac{(2j-1)\theta}{3}\left(\frac{1}{q_0}+\frac{2}{q_1}\right) = \frac{2j-1}{r}$$
as desired.
\end{proof}

To estimate $T_{\ge}(u,v,w)$ in $\widehat{L^r_{xt}}$, we shall use a dyadic decomposition with respect to the $y$-variable
instead of the HLS-inequality.

\begin{lemma}
\label{luis}
Let $1 \le r < \rho \le \infty$. Then
$$\n{T_{\le}(u,v,w)}{\widehat{L^r_{xt}}} \ls \n{u_0}{\widehat{L^{\rho}_{x}}}\n{D_x^{-\frac{2j-1}{2r}}v_0}{\widehat{L^{r}_{x}}}\n{D_x^{-\frac{2j-1}{2r}}w_0}{\widehat{L^{r}_{x}}}.$$
\end{lemma}

\begin{proof}
 We have
$$\F T_{\le}(u,v,w)(\xi,\tau)= c (K^+_{\le}(\xi,\tau) + K^-_{\le}(\xi,\tau)),$$
where
$$K^{\pm}_{\le}(\xi,\tau)=\int_{\{1 \le 2y \le|\xi-\xi_1|\}}  \frac{d\xi_1}{|g'(\xi_1;y)|}\F_xu_0(\xi_1)\F_xv_0(\tfrac{\xi-\xi_1}{2}\pm y)\F_xw_0(\tfrac{\xi-\xi_1}{2}\mp y).$$
Here $y$ is the positive zero of \eqref{g3} as before. By symmetry between $v$ and $w$ it suffices to treat $K^+_{\le}$,
which we decompose dyadically with respect to $y$ to obtain the following upper bound:
\begin{multline*}
 \sum_{k=0}^{\infty}\int_{\{1 \le 2y \le|\xi-\xi_1|\,, \,\,y \sim 2^k\}} 
 \frac{d\xi_1}{|g'(\xi_1;y)|}\F_xu_0(\xi_1)\F_xv_0(\tfrac{\xi-\xi_1}{2}+ y)\F_xw_0(\tfrac{\xi-\xi_1}{2}- y) \\
 \ls \sum_{k=0}^{\infty}2^{-k}\int_{\{y \sim 2^k\}}d\xi_1\F_xu_0(\xi_1)\F_xD_x^{-\frac{2j-1}{2}}v_0(\tfrac{\xi-\xi_1}{2}+
 y)\F_xD_x^{-\frac{2j-1}{2}}w_0(\tfrac{\xi-\xi_1}{2}- y) \\
 \ls \sum_{k=0}^{\infty}2^{-k}\n{u_0}{\widehat{L^{p}_{x}}}\lambda(\{y \sim 2^k\})^{\frac{1}{p}}
 \n{D_x^{-\frac{2j-1}{2}}v_0}{\widehat{L^{1}_{x}}}\n{D_x^{-\frac{2j-1}{2}}w_0}{\widehat{L^{1}_{x}}},
\end{multline*}
where $\lambda(\{y \sim 2^k\})$ denotes the Lebesgue measure of $\{\xi_1 : y(\xi_1) \sim 2^k\}=:K$. We claim that
$\lambda(\{y \sim 2^k\})\ls 2^k$. To see this, we write $K=K_1 \cup K_2$, where in $K_1$ we assume that
$|\xi_1| \ls 2^k $, $|\xi-\xi_1| \ls 2^k $, $|\xi+\xi_1| \ls 2^k $ or $|\xi-3\xi_1| \ls 2^k $. 
Then $K_1$ consists of a finite number of intervals of total length bounded by $c2^k$. To estimate the contribution
coming from $K_2$, we calculate
\begin{multline*}
 \frac{\pd g}{\pd \xi_1}(\xi_1;y)= \frac{2j+1}{4}(\xi +
 \xi_1)(3\xi_1-\xi)\sum_{l=0}^{j-1}\xi_1^{2l}(\tfrac{\xi-\xi_1}{2})^{2(j-l-1)} \\
 -\sum_{l=1}^j\binom{ 2j+1 }{ 2l}(2(j-l)+1)(\tfrac{\xi - \xi_1}{2})^{2(j-l)}y^{2l},
\end{multline*}
so that in $K_2$
$$\left| \frac{\pd g}{\pd \xi_1}(\xi_1;y)\right| \gs y|\xi-\xi_1|(\xi_1^{2(j-1)}+(\xi - \xi_1)^{2(j-1)} )\gs 
\left| \frac{\pd g}{\pd x}(\xi_1;y)\right|.$$
(Observe that $|\xi-\xi_1| \le |\xi+\xi_1| + |\xi-3 \xi_1|$.) This gives
$$\lambda (K_2) = \int_{K_2} d\xi_1 = \int_{K_2} \left|\frac{\pd \xi_1}{\pd \tau}\frac{\pd\tau}{\pd y}\right|dy
= \int_{K_2}\left|\frac{\frac{\pd g}{\pd x}(\xi_1;y)}{\frac{\pd g}{\pd \xi_1}(\xi_1;y)}\right|dy \ls 2^k,$$
and the claim is shown. Hence, for any $p>1$,
\begin{multline}
 \label{220}
\n{K^+_{\le}}{L^{\infty}_{\xi \tau  }} \ls 
\sum_{k=0}^{\infty}2^{-\frac{k}{p'}}\n{u_0}{\widehat{L^{p}_{x}}}\n{D_x^{-\frac{2j-1}{2}}v_0}{\widehat{L^{1}_{x}}}\n{D_x^{-\frac{2j-1}{2}}w_0}{\widehat{L^{1}_{x}}} \\
\ls \n{u_0}{\widehat{L^{p}_{x}}}\n{D_x^{-\frac{2j-1}{2}}v_0}{\widehat{L^{1}_{x}}}\n{D_x^{-\frac{2j-1}{2}}w_0}{\widehat{L^{1}_{x}}}.
\end{multline}
On the other hand, by integration with respect first to $d\tau = g'(\xi_1;y)dy$, to $d\xi$ and finally to $d\xi_1$, 
we see that
\begin{equation}\label{221}
\n{K^+_{\le}}{L^{1}_{\xi \tau  }} \ls \n{u_0}{\widehat{L^{\infty}_{x}}}\n{v_0}{\widehat{L^{\infty}_{x}}}\n{w_0}{\widehat{L^{\infty}_{x}}}.
\end{equation}
Now multilinear interpolation between (\ref{220}) and (\ref{221}) leads to
$$\n{K^+_{\le}}{L^{r'}_{\xi \tau  }} \ls \n{u_0}{\widehat{L^{\rho}_{x}}}\n{D_x^{-\frac{2j-1}{2r}}v_0}{\widehat{L^{r}_{x}}}\n{D_x^{-\frac{2j-1}{2r}}w_0}{\widehat{L^{r}_{x}}},$$
which gives the desired result.

\end{proof}

\begin{kor}
\label{xluis}
Let $1\le r < \rho \le \infty$, $\beta>\frac{1}{\rho}$, $b>\frac{1}{r}$ and $\e>0$. Then
$$\n{T_{\le}(u_1,u_2,u_3)}{\widehat{L^r_{xt}}} \ls \n{u_1}{X^{\rho}_{0,\beta}}\n{D_x^{-\frac{2j-1}{2r}}u_2}{X^{r}_{0,b}}\n{D_x^{-\frac{2j-1}{2r}}u_3}{X^{r}_{0,b}}$$
and
$$\n{T_{\le}(u_1,u_2,u_3)}{\widehat{L^r_{xt}}} \ls \n{u_1}{X^{r}_{\e,b}}\n{D_x^{-\frac{2j-1}{2r}}u_2}{X^{r}_{0,b}}\n{D_x^{-\frac{2j-1}{2r}}u_3}{X^{r}_{0,b}}$$
are valid.
\end{kor}

\section{$\x$-Estimates for the mKdV hierarchy}
\label{nonlin}
\subsection{Cubic nonlinearities}
This subsection is devoted to the proof of the following estimate.

\begin{theorem}
 \label{cubic}
Let $l_0,\dots,l_3 \ge 0$ with $\sum_{i=0}^3l_i =2j-1$, $2 \ge r > 1$ and $s \ge s_j(r)= \frac{2j-1}{2r'}$. Then for all
$b' < \tfrac{1}{r'}(\tfrac{1}{2j+1}-\tfrac{1}{2})$ and $b > \frac{1}{r}$ the estimate
\begin{equation}
\label{40}
\n{\partial _x^{l_0} (\prod_{i=1}^3 \partial _x^{l_i}u_i)}{X^{r}_{s,b'}}\ls \prod_{i=1}^3 \n{u_i}{X^r_{s,b}}
\end{equation}
holds true.
\end{theorem}

Preparations: Without loss of generality we may assume that $s = s_j(r)$. Then we rewrite the left hand side of \eqref{40} as
$$\n{\langle \tau -  \xi^{2j+1} \rangle^{b'} \langle  \xi \rangle^{s}\xi ^{l_0}\int d \nu  \prod_{i=1}^3 \xi_i ^{l_i}\F{u_i}(\xi_i,\tau_i)}{L^{r'}_{\xi,\tau}},$$
where $d \nu = d\xi_1 d\xi_{2} d\tau_1 d \tau_{2}$ and $\sum_{i=1}^3 (\xi_i,\tau_i) = (\xi, \tau)$.
We shall use the following notation:
\begin{itemize}
\item $\xi_{max}$, $\xi_{med}$, $\xi_{min}$ are defined by $|\xi_{max}| \ge |\xi_{med}| \ge |\xi_{min}|$,
\item $p$ denotes the projection on low frequencies, i. e. $\F{pf}(\xi)=\chi_{\{|\xi| \le 1\}}\F{f}(\xi)$,
\item $f \preceq g$ is shorthand for $|\F{f}| \ls |\F{g}|$,
\item for the mixed weights coming from the $\x$ - norms we shall write $\sigma_0 := \tau - \xi^{2j+1}$ and $\sigma_i := \tau_i - \xi_i^{2j+1}$, $1 \le i \le 3$, respectively,
\item the Fourier multiplier associated with these weights is denoted by $\Lambda ^b := \F ^{-1} \langle \tau - \xi^{2j+1} \rangle ^b \F$,
\item $J^s=\F_x ^{-1} \lb \xi \rb^s \F_x$ is the Bessel potential operator of order $-s$, where as usual $\lb \xi \rb^s=(1 + \xi^2)^{\frac{s}{2}}$,
\end{itemize}
Allthough our argument relies almost completely on the smoothing effects inherent in the estimates for free solutions
obtained in the previous section, we will need sometimes the following resonance relation for the phase function $\phi(\xi)=\xi^{2j+1}$ and a cubic nonlinearity.
\begin{lemma}
 \label{resonance}
$$(\xi_1+\xi_2+\xi_3)^{2j+1}-\xi_1^{2j+1}-\xi_2^{2j+1}-\xi_3^{2j+1}=(\xi_1 +\xi_2)(\xi_2 +\xi_3)(\xi_3 +\xi_1)\Sigma ,$$
where
$$\Sigma \gs (\xi_1 +\xi_2)^{2j-2}+(\xi_2 +\xi_3)^{2j-2}+(\xi_3 +\xi_1)^{2j-2}.$$
Consequently,
\begin{multline*}
 |(\xi_1 +\xi_2)(\xi_2 +\xi_3)(\xi_3 +\xi_1)|((\xi_1 +\xi_2)^{2j-2}+(\xi_2 +\xi_3)^{2j-2}+(\xi_3 +\xi_1)^{2j-2})\\
\ls \sum_{i=0}^3 \lb \sigma_i \rb \ls \max_{i=0}^3 \lb \sigma_i \rb \ls \prod_{i=0}^3 \lb \sigma_i \rb. \hspace{5cm}
\end{multline*}

\end{lemma}

\begin{proof}
 With $x=\xi_1 +\xi_2$, $y=\xi_2 +\xi_3$, and $z=\xi_3 +\xi_1$, the left hand side of the claimed identity becomes
$$\frac{1}{2^{2j+1}}\left\{ (x + y + z)^{2j+1} - (x + y - z)^{2j+1} - ((x - y + z)^{2j+1} - (x - y - z)^{2j+1})\right\},$$
with
\begin{multline*}
\left\{ \dots \right\} = \sum_{k=0}^{2j+1}\binom{2j+1}{k}\left\{ (x+y)^{2j+1-k}(z^k -(-z)^k) -  (x-y)^{2j+1-k}(z^k -(-z)^k)\right\} \\
= 2z\sum_{l=0}^{j}\binom{2j+1}{2l+1}((x+y)^{2(j-l)} - (x-y)^{2(j-l)})z^{2l}\qquad\qquad \qquad\\
= 2z\sum_{l=0}^{j}\binom{2j+1}{2l+1}z^{2l}\sum_{k=0}^{2(j-l)}\binom{2(j-l)}{k}x^{2(j-l)-k}(y^k -(-y)^k) \qquad\\
= 4xyz\sum_{l=0}^{j-1}\binom{2j+1}{2l+1}z^{2l}\sum_{m=0}^{j-l-1}\binom{2(j-l)}{2m+1}x^{2(j-l-m-1)}y^{2m}.\qquad \qquad \qquad
\end{multline*}
Now we pic out the contributions for $l=j-1$, $l=0, m=j-1$, and $l=m=0$.
\end{proof}

\begin{proof}[Proof of Theorem \ref{cubic}]
 Apart from the trivial region where $|\xi_{max}| \le 1$, whose contribution can be estimated by
$$\n{\prod_{i=1}^3 p u_i}{\widehat{L^r_{xt}}}\ls \prod_{i=1}^3 \n{p u_i}{\widehat{L^{3r}_{xt}}}\ls \prod_{i=1}^3 \n{p u_i}{X^{r}_{0,b}}\ls \prod_{i=1}^3 \n{u_i}{X^{r}_{s,b}},$$
we consider two main cases:

\begin{itemize}
\item[1.] The nonresonant case, where $|\xi_{max}| \gg |\xi_{min}|$, and
\item[2.] the resonant case, where $|\xi_{max}| \sim |\xi_{min}|$.
\end{itemize}

1. In the nonresonant case we assume without loss of generality that $|\xi_1|\ge|\xi_2|\ge|\xi_3|$. We distingiush two
subcases. \\

Subcase 1.1: Here we assume additionally $|\xi_1| \ge 1.1 |\xi_2|$, such that $|\xi_1+\xi_2||\xi_1-\xi_2| \gs \xi_1^2$
as well as $|\xi_1+\xi_2||\xi_1+\xi_2+ 2\xi_3| \gs \xi_1^2$. Then for the symbols of the Fourier multipliers $M_{r,j}(u_1,u_2)$
and $M^*_{\rho ',j}(u_1 u_2,u_3)$ we have that $m_j(\xi_1,\xi_2) \gs \xi_1^{2j}$ and $m^*_j(\xi_1+\xi_2,\xi_3) \gs \xi_1^{2j}$.
This leads to a gain of $\frac{2j}{r}$ derivatives in the application of Corollary \ref{xBill} and of $\frac{2j}{\rho'}$ derivatives in the application of \eqref{204}, both on the highest frequency. More precisely we have
\begin{eqnarray*}
J^s \partial_x^{l_0}(\partial_x^{l_1}u_1\partial_x^{l_2}u_2\partial_x^{l_3}u_3) & \preceq & 
 \partial_x^{l_0}(J^s\partial_x^{l_1}u_1\partial_x^{l_2}u_2\partial_x^{l_3}u_3)\\
& \preceq & M_{r,j}(J^{s+\frac{2j}{r'}-1}u_1,u_2)u_3 \\
& \preceq & M^*_{\rho ',j}\left(M_{r,j}(J^{s}u_1,u_2),J^{2j(\frac{1}{r'}-\frac{1}{\rho'})-1}u_3 \right) .
\end{eqnarray*}
Now the dual version \eqref{204} of the bilinear estimate is applied to obtain
\begin{multline*}\|M^*_{\rho ',j}\left(M_{r,j}(J^{s}u_1,u_2),J^{2j(\frac{1}{r'}-\frac{1}{\rho'})-1}u_3 \right)\|_{X^{r}_{0,b'}}\\
\ls \|M_{r,j}(J^{s}u_1,u_2)\|_{\widehat{L^r_{xt}}}\|J^{2j(\frac{1}{r'}-\frac{1}{\rho'})-1}u_3\|_{X^{\rho '}_{0,-b'}}.
\end{multline*}
This requires $b' < - \frac{1}{\rho '}$ (condition 1). For the first factor we obtain from Corollary \ref{xBill} that
$$\|M_{r,j}(J^{s}u_1,u_2)\|_{\widehat{L^r_{xt}}} \ls \n{u_1}{X^{r}_{s,b}}\n{u_2}{X^{r}_{0,b}}\ls \n{u_1}{X^{r}_{s,b}}\n{u_2}{X^{r}_{s,b}},$$
while Sobolev type embeddings give
$$\|J^{2j(\frac{1}{r'}-\frac{1}{\rho'})-1}u_3\|_{\X{\rho '}{0}{-b'}} \ls \n{u_3}{X^{r}_{sb}}$$
for the second factor, provided that $b+b'-\frac{1}{r} > -\frac{1}{\rho'}$ (condition 2) and $s > \frac{2j-1}{r'}-\frac{2j+1}{\rho '}$ (condition 3). Finally we choose $\rho$ with $-\frac{1}{\rho'} \in (b', \min{(b+b'-\frac{1}{r},\frac{1}{r'}(\tfrac{1}{2j+1}-\tfrac{1}{2})})$, such that the conditions 1, 2, and 3 are satisfied. \\
 
Subcase 1.2: Here we have $|\xi_1| \le 1.1 |\xi_2|$ and hence $|\xi_2| \gg |\xi_3|$, so that $|\xi_1+\xi_3||\xi_1-\xi_3| \gs \xi_1^2$
as well as $|\xi_1+\xi_3||\xi_1+\xi_3+ 2\xi_2| \gs \xi_1^2$. So the argument performed in subcase 1.1 applies with the same
upper bound, if $u_2$ and $u_3$ are interchanged. \\

2. In the resonant case we distinguish again several subcases or -regions, respectively: \\

Subcase 2.1: At least for one pair $(i,j)$ we have $|\xi_i-\xi_j|\ge|\xi_i+\xi_j|$.
Here we may assume by symmetry that $|\xi_2-\xi_3|\ge|\xi_2+\xi_3|$. For the contribution of this region
we have
$$J^s \partial_x^{l_0}(\partial_x^{l_1}u_1\partial_x^{l_2}u_2\partial_x^{l_3}u_3) \preceq T_{\ge}(J^{s+s_0}u_1,J^{s+s_1}u_2,J^{s+s_1}u_3),$$
whenever $s_{0,1} \ge 0$ fulfill $s_0 + 2s_1 = \frac{2j-1}{r}$. Now \eqref{xTrill} of Corollary \ref{Trill}
gives the desired bound.\\

Subcase 2.2:  $|\xi_1-\xi_2|\le|\xi_1+\xi_2|$, $|\xi_2-\xi_3|\le|\xi_2+\xi_3|$ and $|\xi_3-\xi_1|\le|\xi_3+\xi_1|$, so that all the $\xi_i$ have the same sign, which implies by Lemma \ref{resonance} that
$$|\xi_1|^{2j+1} \sim |\xi_2|^{2j+1} \sim |\xi_3|^{2j+1} \le \prod_{i=0}^3 \langle \sigma_i \rangle.$$ \\

Subsubcase 2.2.1: Let us assume first, that at least one of the $|\xi_i-\xi_j|\ge 1$.
By symmetry we may restrict ourselves to $|\xi_2-\xi_3|\ge 1$. Gaining a $\langle \xi \rangle ^{\e}$ from the $\sigma's$ we obtain as an upper bound for the contribution from this subregion
$$\n{T_{\le}(J^{s-\e}\Lambda^{\e}u_1,J^{j-\frac{1}{2}}\Lambda^{\e}u_2,J^{j-\frac{1}{2}}\Lambda^{\e}u_3)}{\widehat{L^r_{xt}}}   \ls \prod_{i=1}^3 \n{u_i}{X^{r}_{s,b}},$$
where we have used the second part of Corollary \ref{xluis}. ($\e$ was chosen sufficiently small here, for given $b>\frac{1}{r}$.)\\

Subsubcase 2.2.2: Now we consider $|\xi_i-\xi_j|\le 1$ for all $1 \le i , j \le 3$.
Again, we can gain a $\langle \xi \rangle ^{\e}$ from the $\sigma's$. Now, writing 
$$f_i(\xi,\tau)=\langle \xi \rangle^s \langle \tau- \xi^{2j+1}\rangle^b \F u_i(\xi,\tau), \,\,\,\,1\le i\le 3, \,\,\,\,\mbox{such that}\,\,\,\,
\n{f_i}{L^{r'}_{\xi \tau}}=\n{u_i}{X^{r}_{sb}},$$ 
it suffices to show 
\begin{equation}\label{300}
\n{\langle \xi \rangle^{s+l_0-\e} \int_A d\nu \prod_{i=1}^3\langle \xi_i \rangle^{l_i-s}\langle \tau_i- \xi_i^{2j+1}\rangle^{-\tilde{b}}f_i(\xi_i,\tau_i)}{L^{r'}_{\xi \tau}} \ls \prod_{i=1}^3 \n{f_i}{L^{r'}_{\xi \tau}},
\end{equation}
where in $A$ all the differences $|\xi_k-\xi_j|$, $1 \le k , j \le 3$, are bounded by $1$ and $|\xi| \sim |\xi_i| \sim \langle \xi_i \rangle$ for all $1\le i\le 3$. Here we have replaced $b$ by a slightly smaller $\tilde{b}$, so that still $\tilde{b}r>1$. By H\"older's inequality and Fubini's Theorem the proof of \eqref{300} is reduced to show that
\begin{equation}\label{301}
\sup_{\xi, \tau}\langle \xi \rangle^{2j-1-2s-\e} \left( \int_A d\nu \prod_{i=1}^3 \langle \tau_i- \xi_i^{2j+1}\rangle^{-\tilde{b}r}\right)^{\frac{1}{r}} < \infty .
\end{equation}
Using \cite[Lemma 4.2]{GTV97} twice, we see that
$$\int_A d\nu \prod_{i=1}^3 \langle \tau_i- \xi_i^{2j+1}\rangle^{-\tilde{b}r} \ls \int_{A'} d\xi_1d\xi_2 \langle \tau-  \xi_1^{2j+1} - \xi_2^{2j+1} - \xi_3^{2j+1}\rangle^{-\tilde{b}r},$$
where $A'$ is simply the projection of $A$ onto $\mathbb{R}^2$. We decompose
$$A'=A_0 \cup A_1 \cup \bigcup_{0\le k,l \ls \ln{(|\xi|)}}A_{kl},$$
where in $A_0$ ($A_1$) we have $|\xi_1-\xi_3|\ls |\xi|^{1-2j}$ ($|\xi_2-\xi_3|\ls |\xi|^{1-2j}$),
so that the contributions of these subregions are bounded by $c|\xi|^{1-2j}$. In $A_{kl}$ we have $|\xi_1-\xi_3|\sim 2^{-k}$
and $|\xi_2-\xi_3|\sim 2^{-l}$, where by symmetry we may assume $l \le k$. Now, with $g$ as in \eqref{g3} and
$y=\xi_2-\tfrac{\xi-\xi_1}{2}$, we have for the integral over $A_{kl}$
\begin{multline*}
 \int_{A_{kl}} d\xi_1d\xi_2 \langle \tau-  \xi_1^{2j+1} - \xi_2^{2j+1} - \xi_3^{2j+1}\rangle^{-\tilde{b}r}
= \int_{\substack{y \sim 2^{-l}\\|\xi_1-\xi_3|\sim 2^{-k}}} d\xi_1 dy \lb g(\xi_1;y)\rb^{-\tilde{b}r} \\
\ls \int_{\substack{y \sim 2^{-l}\\|\xi_1- \tfrac{\xi}{3}|\ls 2^{-l}}}   \frac{d\xi_1 dg}{|g'(\xi_1;y)|}\lb g\rb^{-\tilde{b}r} \ls |\xi|^{1-2j}2^l\int_{|\xi_1- \tfrac{\xi}{3}|\ls 2^{-l}}d\xi_1 \int dg \lb g\rb^{-\tilde{b}r} \ls |\xi|^{1-2j}.
\end{multline*}
Summing up over $k$ and $l$ we find
$$\int_{A'} d\xi_1d\xi_2 \langle \tau-  \xi_1^{2j+1} - \xi_2^{2j+1} - \xi_3^{2j+1}\rangle^{-\tilde{b}r}
\ls \frac{(\ln{|\xi|})^2}{|\xi|^{2j-1}} ,$$
which gives \eqref{301}, since $2j-1-2s-\frac{2j-1}{r} \le 0$.
\end{proof}

\subsection{Quintic and higher nonlinearities}

The quintic and higher order contributions to the nonlinearities can be treated in a uniform way and with a better
lower bound on $s$ than for the cubic terms discussed above. This is the subject of the present section, where we shall
prove the following estimate.

\begin{theorem}
 \label{quintic+}
Let $j \ge k \ge 2$ and $\sum_{i=0}^{2k+1}l_i = 2(j-k)+1$, where $l_i \ge 0$. Then for $2 \ge r >1$ and $s > - \frac{1}{r'}$ there exists
$b'> - \frac{1}{r'}$, such that for all $b > \frac{1}{r}$
\begin{equation}
 \label{quinticest}
 \|\pd_x^{l_0}\prod_{i=1}^{2k+1}\pd_x^{l_i}u_i\|_{X^r_{s,b'}} \ls \prod_{i=1}^{2k+1}\|u_i\|_{X^r_{s,b}}.
\end{equation}
\end{theorem}

Theorem \ref{quintic+} will be obtained by multilinear interpolation. First we prove the special case of the Theorem,
where $r=2$, see the next subsection. This special result for $L^2$-based spaces leads to an interesting by product
of our analysis for those equations having only quintic or higher order terms in their nonlinear part (Proposition \ref{nocubic}
in the introduction). Then we will show a relatively weak version of the estimate for $r$ in a small intervall $(1, r_0)$
and large values of $s$, more precisely we will assume $s > \frac{2(j-k)+1}{2kr'}=:s_0(r)$, see Theorem \ref{quintic+r}
below. Since $s_0(r) \rightarrow 0$ for $r \rightarrow 1$ this estimate is nonetheless sufficient to obtain the full
result by interpolation.

\subsubsection{$H^s$-estimates}

\begin{theorem}
 \label{quintic+Hs}
Let $j \ge k \ge 2$ and $\sum_{i=0}^{2k+1}l_i = 2(j-k)+1$, where $l_i \ge 0$. Then for $s > - \frac12$ there exists
$b'> - \frac12$, such that for all $b > \frac12$
\begin{equation}
 \label{quinticestHs}
 \|\pd_x^{l_0}\prod_{i=1}^{2k+1}\pd_x^{l_i}u_i\|_{X_{s,b'}} \ls \prod_{i=1}^{2k+1}\|u_i\|_{X_{s,b}}.
\end{equation}
\end{theorem}

The proof of Theorem \ref{quintic+Hs} relies substantially on the interplay between the local smoothing estimate
\begin{equation}
 \label{Kato}
\|\pd _x^j e^{\pm t \pd _x^{2j+1}}u_0\|_{L_x^{\infty}L_t^2} = c \|u_0\|_{L_x^2}
\end{equation}
from \cite[Theorem 4.1]{KPV91} and the maximal function estimate
\begin{equation}
 \label{max}
 \|e^{\pm t \pd _x^{2j+1}}u_0\|_{L_x^4L_t^{\infty}} \ls \|D_x^{\frac14}u_0\|_{L_x^2},
\end{equation}
cf. \cite[Theorem 3]{S87} and \cite[Theorem 2.5]{KPV91}. Taking care of low frequency issues by Sobolev-type embeddings
for $L_x^pL_t^q$-spaces, see \cite[Lemma 3.15 and its proof]{KPV93}, and
interpolating a) with the space-time Sobolev-inequality $\|u\|_{L_{xt}^{\infty}} \ls \|u\|_{X_{\sigma b}}$,
$\sigma, b > \frac12$, as well as b) with the trivial case $X_{0,0}=L^2_{xt}$, we obtain from \eqref{Kato} the following
$X_{s,b}$-estimates.

\begin{lemma}\
 \label{LemmaxKato}
\hfill
\begin{enumerate}
 \item[a)] Let $b > \frac12$, $2 \le p \le \infty$, and $s > \frac12 - \frac{2j+1}{p}$. Then
\begin{equation}
 \label{xKatoSob}
\| u\|_{L_x^{\infty}L_t^p} \ls \|u\|_{X_{s,b}},
\end{equation}
 \item[b)] Let $2 \le p \le \infty$, $s=j(1-\frac{2}{p})$ and $b > \frac12 - \frac{1}{p}$. Then
\begin{equation}
 \label{xKato}
\|J^s u\|_{L_x^{p}L_t^2} \ls \|u\|_{X_{0,b}}
\end{equation}
and, by duality,
\begin{equation}
 \label{xKato*}
\|u\|_{X_{s,-b}} \ls \| u\|_{L_x^{p'}L_t^2}.
\end{equation}
\end{enumerate}
\end{lemma}

Similarly, the following lemma can be derived from \eqref{max}.

\begin{lemma}
 \label{Lemmaxmax}
\hfill
\begin{enumerate}
\item[a)] Let $b > \frac12$, $4 \le p \le \infty$, and $s > \frac12 - \frac{1}{p}$. Then
\begin{equation}
 \label{xmaxSob}
\| u\|_{L_x^{p}L_t^{\infty}} \ls \|u\|_{X_{s,b}}
\end{equation}
 \item[b)] For $2 \le p \le 4$, $2 \le q \le \infty$ with $\frac{1}{p}=\frac14 + \frac{1}{2q}$, $s=\frac14 - \frac{1}{2q}$,
  and $b > \frac12 - \frac{1}{q}$ we have
\begin{equation}
 \label{xmax}
\| u\|_{L_x^{p}L_t^{q}} \ls \|u\|_{X_{s,b}}
\end{equation}
and
\begin{equation}
 \label{xmax*}
\|u\|_{X_{0,-b}} \ls \|J^s u\|_{L_x^{p'}L_t^{q'}}.
\end{equation}
\end{enumerate}
\end{lemma}

Now we are prepared to prove the decisive estimate for quintic and higher nonlinearities in the case of $L^2$-based spaces.

\begin{proof}[Proof of Theorem \ref{quintic+Hs}]
 Let $\xi_i$ denote the frequencies of $u_i$, $1 \le i \le 2k+1$. We assume $|\xi_1|\ge|\xi_2|\ge \dots \ge |\xi_{2k+1}|$ and
distingish two cases.

\quad 

1. $|\xi_1| \sim |\xi|$. In this case, for any $\delta \ge 0$ and  $\e >0$, the left hand side of \eqref{quinticestHs} is bounded
by
$$\|J^{j-\e}( J^{j-2k+1+\e+s+\delta}u_1\!\! \prod_{i=2}^{2k+1}\!\!J^{-\frac{\delta}{2k}}u_i)\|_{X_{0,b'}}
\ls \|J^{j-2k+1+\e+s+\delta}u_1 \!\!\prod_{i=2}^{2k+1}\!\!J^{-\frac{\delta}{2k}}u_i\|_{L_x^{1+\e'}L_t^2}$$
by \eqref{xKato*}. Here $\e'$ depends on $\e$ and both can be made arbitrarily small by choosing $b'$ close enough to $-\frac12$.
H\"older's inequality gives the upper bound
$$\|J^{j-2k+1+\e+s+\delta}u_1\|_{L_x^{\infty}L_t^2}\prod_{i=2}^{2k+1}\|J^{-\frac{\delta}{2k}}
u_i\|_{L_x^{2k(1+\e')}L_t^{\infty}}.$$
For the first factor we have by \eqref{xKato}
$$\|J^{j-2k+1+\e+s+\delta}u_1\|_{L_x^{\infty}L_t^2}\ls \|u_1\|_{X_{s,b}},$$
provided $b>\frac12$ and $\delta + 1 < 2k$ (condition 1). For the other factors we use \eqref{xmaxSob} to get
$$\prod_{i=2}^{2k+1}\|J^{-\frac{\delta}{2k}}u_i\|_{L_x^{2k(1+\e')}L_t^{\infty}} \ls
\prod_{i=2}^{2k+1}\|u_i\|_{X_{s,b}},$$
as long as $b>\frac12$ and $s > \frac12 - \frac{\delta +1}{2k}$ (condition 2). Choosing
$\delta \in (2k(\tfrac12-s)-1,2k-1)$ both conditions, 1 and 2, are fulfilled.

\quad

2. $|\xi_1| \sim |\xi_2|$. Here the contribution is bounded by
\begin{multline*}
 \|J^{-\frac14}( J^{j+s}u_1J^{j-2k+1+ \frac14+\delta}u_2 \prod_{i=3}^{2k+1}J^{-\frac{\delta}{2k-1}}u_i)\|_{X_{0,b'}}\\
\ls \|J^{j+s}u_1J^{j-2k+1+ \frac14+\delta}u_2 \prod_{i=3}^{2k+1}J^{-\frac{\delta}{2k-1}}u_i\|_{L_x^{\frac43+\e}L_t^{1+\e'}}
\end{multline*}
by \eqref{xmax*}. Here again $\e$ and $\e'$ can be made arbitrarily small by choosing $b'$ close to $-\frac12$. $\delta \ge 2$
is to be fixed later on. Now H\"older's inequality gives the upper bound
$$\|J^{j+s}u_1\|_{L_x^{\infty}L_t^2}\|J^{j-2k+1 + \frac14+\delta}u_2\|_{L_x^{\infty}L_t^{2+\e''}}
\prod_{i=3}^{2k+1}\|J^{-\frac{\delta}{2k-1}}u_i\|_{L_x^{\frac{(4+3\e)(2k-1)}{3}}L_t^{\infty}}.$$ 
The first factor is bounded by $c\|u_1\|_{X_{sb}}$ by \eqref{xKatoSob}, which also shows that
$$\|J^{j-2k+1 + \frac14+\delta}u_2\|_{L_x^{\infty}L_t^{2+\e''}} \ls \|u_2\|_{X_{s,b}},$$
provided that $b>\frac12$ and $s > \frac54 + \delta -2k$ (condition 3). For the other factors we use \eqref{xmaxSob} to obtain
$$\prod_{i=3}^{2k+1}\|J^{-\frac{\delta}{2k-1}}u_i\|_{L_x^{\frac{(4+3\e)(2k-1)}{3}}L_t^{\infty}}
\ls \prod_{i=3}^{2k+1}\|u_i\|_{X_{s,b}},$$
which requires again $b>\frac12$ as well as $s > \frac12 - \frac{3}{4(2k-1)}-\frac{\delta}{2k-1}$ (condition 4).
Both conditions 3 and 4 can be satisfied by choosing $\delta \in ((2k-1)(\frac12-s)-\frac34,2k+s-\frac54)$, which is not
empty, since $s > - \frac12$.
\end{proof}

\subsubsection{Estimates close to the critical space}

\begin{theorem}
 \label{quintic+r}
Let $j \ge k \ge 2$ and $\sum_{i=0}^{2k+1}l_i = 2(j-k)+1$, where $l_i \ge 0$. Then there exists $r_0 > 1$ such that for all 
$r \in (1, r_0)$ and $s >  \frac{2(j-k)+1}{2kr'}$ there exists $b'> - \frac{1}{r'}$, such that for all $b > \frac{1}{r}$ the
estimate \eqref{quinticest} holds true.
\end{theorem}

\begin{proof}

In general we assume that $|\xi_1|\ge|\xi_2|\ge \dots \ge |\xi_{2k+1}|$. First we consider the case, where there are at
least four high frequency factors, i. e. $|\xi_4| \gs |\xi_1|$. We fix $s_1 > \frac14(2(j-k)+1+s+(2k-3)(\frac{1}{r}-s))$
and $s_2 < s-\frac{1}{r}$ so that $4s_1+(2k-3)s_2=2(j-k)+1+s$. (Since we consider $r$ close to $1$, we may assume $s<\frac{1}{r}$
here.) Then the contribution of this region to the left hand side of \eqref{quinticest} is controlled by
\begin{multline*}
\qquad \|J^{s_1}u_1  \cdots  J^{s_1}u_4 \cdot J^{s_2}u_5 \cdots  J^{s_2}u_{2k+1}\|_{\widehat{L^r_{xt}}}\\
\ls \|J^{s_1}u_1 \cdots J^{s_1}u_4 \cdot J^{s_2}u_5 \cdots  J^{s_2}u_{2k+1}\|_{L^r_{xt}}\ls
\prod_{i=1}^4\|J^{s_1}u_i\|_{L^{4r}_{xt}}\prod_{i=5}^{2k+1}\|J^{s_2}u_i\|_{L^{\infty}_{xt}}
\end{multline*}
by the Hausdorff-Young and H\"older inequalities. Our choice of $s_2$ implies for the last factors
$$\prod_{i=5}^{2k+1}\|J^{s_2}u_i\|_{L^{\infty}_{xt}} \ls \prod_{i=5}^{2k+1}\|u_i\|_{X^r_{s,b}}.$$
Concerning the first four factors we observe that by \eqref{xFeSt} and Sobolev type embeddings the inequality
$$\|u\|_{L^{4r}_{xt}} \ls \|u\|_{X^r_{\sigma ,b}}$$
holds, provided $b>\frac{1}{r}$ and $\sigma > \frac{2-2j}{4r}$. This gives the desired
$$\prod_{i=1}^4\|J^{s_1}u_i\|_{L^{4r}_{xt}} \ls \prod_{i=1}^{4}\|u_i\|_{X^r_{s,b}},$$
as long as $s>s_1 +\frac{2-2j}{4r}$, that is for $s>\frac{2(j-k)+1}{2kr'}$, as assumed.
This concludes the discussion of the case, where at least four factors have high frequencies.
So we may assume henceforth, that there are at least two low frequency factors, i.e., we suppose now
$|\xi_{2k}|\ll |\xi_1|$. Two subcases are distinguished.

\quad

1. $|\xi_{2}|\ll |\xi_1|$. Using again the notation $f \preceq g$ for $|\F f| \ls |\F g|$ we have here
$$\pd_x^{l_0}\prod_{i=1}^{2k+1}\pd_x^{l_i}u_i \preceq \left(M_{r,j}(J^{2(j-k)+1-\frac{2j}{r}}u_1,u_{2k+1})u_2 \cdots u_{2k-1} \right) \cdot u_{2k}=:u \cdot v,$$
where the first factor $u$ has the frequency $\xi-\xi_{2k}$, for which $|\xi-\xi_{2k}| \gg |\xi_{2k}|$. Hence, for any
$\rho \in (1,r)$ and sufficiently small $\e >0$
\begin{multline*}
 \pd_x^{l_0}\prod_{i=1}^{2k+1}\pd_x^{l_i}u_i \preceq M^*_{\rho',j} \left(M_{r,j}(J^{2(j-k)+1-\frac{2j}{r}-\frac{2j}{\rho'}}u_1,u_{2k+1})u_2 \cdots u_{2k-1}, u_{2k}\right) \\
\preceq M^*_{\rho',j} \left(M_{r,j}(J^{\frac{2(j-k)+1}{r'}-\frac{2j}{\rho'}}u_1,u_{2k+1})J^{-\frac{1}{r}-\e}u_2 \cdots J^{-\frac{1}{r}-\e}u_{2k-1}, J^{-\frac{1}{r}+\frac{1}{\rho'}-\e}u_{2k}\right).
\end{multline*}
Now we choose $\rho$ close enough to $r$, so that the total number of derivatives on $u_1$ becomes nonpositive. Then for 
$b'<-\frac{1}{\rho'}$
\begin{eqnarray*}
 &&\|\pd_x^{l_0}\prod_{i=1}^{2k+1}\pd_x^{l_i}u_i \|_{X^r_{0,b'}}\\
&\ls & \|M^*_{\rho',j} \left(M_{r,j}(u_1,u_{2k+1})J^{-\frac{1}{r}-\e}u_2 \cdots J^{-\frac{1}{r}-\e}u_{2k-1}, J^{-\frac{1}{r}+\frac{1}{\rho'}-\e}u_{2k}\right)\|_{X^r_{0,b'}}\\
&\ls & \|M_{r,j}(u_1,u_{2k+1})J^{-\frac{1}{r}-\e}u_2 \cdots J^{-\frac{1}{r}-\e}u_{2k-1}\|_{\widehat{L^r_{xt}}}
\|J^{-\frac{1}{r}+\frac{1}{\rho'}-\e}u_{2k}\|_{X^{\rho'}_{0,-b'}}\\
&\ls & \|M_{r,j}(u_1,u_{2k+1})\|_{\widehat{L^r_{xt}}}\prod_{i=2}^{2k-1}\|J^{-\frac{1}{r}-\e}u_i\|_{\widehat{L^{\infty}_{xt}}}
\|J^{-\frac{1}{r}+\frac{1}{\rho'}-\e}u_{2k}\|_{X^{\rho'}_{0,-b'}},
\end{eqnarray*}
where we have used \eqref{204} and the trivial endpoint of the Hausdorff-Young inequality. Finally Corollary \ref{xBill}
(for the first) and Sobolev type embeddings (for the other factors) give the upper bound
$$\|\pd_x^{l_0}\prod_{i=1}^{2k+1}\pd_x^{l_i}u_i \|_{X^r_{0,b'}} \ls \prod_{i=1}^{2k+1}\|u_i\|_{X^r_{0,b}}.$$
Thus we have shown that the desired estimate for the contribution from this subcase holds true for $s=0$ and hence
also for any $s>0$.

\quad

2. $|\xi_{2}|\sim |\xi_1|$. If additionally $||\xi|-|\xi_2|| \le \frac12 |\xi_2|$ and hence $|\xi|\sim |\xi_1|$, we can choose
$u$ and $v$ as in subcase 1., so that the symbol of the multiplier $M^*_{\rho',j}(u,v)$ behaves $\sim |\xi_1|^{\frac{2j}{\rho'}}$
and we may argue as before. If on the contrary $||\xi|-|\xi_2|| \ge \frac12 |\xi_2|$, we choose 
$u=M_{r,j}(u_1,u_{2k+1})u_3 \cdots u_{2k}$ with frequency $\xi - \xi_2$ and $v=u_2$ with frequency $\xi_2$. Then
the symbol of $M^*_{\rho',j}(u,v)$ becomes $(|\xi - \xi_2||\xi + \xi_2|((\xi - \xi_2)^{2j-2}+ \xi_2^{2j-2}))^{\frac{1}{\rho'}}
\sim |\xi_1|^{\frac{2j}{\rho'}}$. So we may again argue as in subcase 1. with $u_2$ and $u_{2k}$ interchanged.

\end{proof}

\section{A counterexample - proof of Proposition \ref{counter}}
\label{contra}

In \cite{KPV01} the two parameter family
$$u_{N\omega}(x,t)=\sqrt{6}\omega e^{i(t(N^3-N\om^2)+Nx)}\mbox{sech}(\om(x-(\om^2-3N^2)t))$$
of solutions of a complex mKdV equation was used to show that the Cauchy problem for this equation is ill-posed in $H^s(\R)$
in the $C^0$-uniform sense, if $s<\frac14$. Here we are going to use essentially the same family of functions - with a slightly
different choice of parameters, depending on $j$ - to prove Proposition \ref{counter}. The first step is to modify the $u_{N\om}$
appropriately, so that they solve an equation of type \eqref{complex}.

\begin{lemma}
 \label{solution}
Let $f(z) = {\rm{sech}}(z)$,
\begin{eqnarray*}
 \delta=\sum_{n=0}^j (-1)^n \binom{2j+1}{2n}N^{2(j-n)+1}\om^{2n},\\
 c = \sum_{n=0}^j (-1)^{n+1} \binom{2j+1}{2n+1}N^{2(j-n)}\om^{2n},
\end{eqnarray*}
and
$$v_{N\om}(x,t)=\om e^{i(Nx+\delta t)}f(\om(x-ct)).$$
Then, for a specific choice of the coefficients $a_{jkl}$, $v_{N\om}$ is a solution of \eqref{complex}.
\end{lemma}

\begin{proof}
 We define
$$u(x,t)=u_M(x,t)=e^{i(Mx+\delta_0 t)}f(x-c_0t)$$
with
$$ \delta_0=\sum_{n=0}^j (-1)^n \binom{2j+1}{2n}M^{2(j-n)+1}\quad \mbox{and}\quad
 c_0 = \sum_{n=0}^j (-1)^{n+1} \binom{2j+1}{2n+1}M^{2(j-n)}.$$
Then
\begin{equation}
 \label{c1}
\pd_tu(x,t)=e^{i(Mx+\delta_0 t)}(i\delta_0f(x-c_0t)-c_0f'(x-c_0t))
\end{equation}
and
\begin{multline}
 \label{c2}
\pd_x^{2j+1}u(x,t)=e^{i(Mx+\delta_0 t)}\sum_{k=0}^{2j+1}\binom{2j+1}{k}(iM)^{2j+1-k}f^{(k)}\\
=(-1)^{j+1}e^{i(Mx+\delta_0 t)}\sum_{n=0}^{j}(-1)^{n+1}M^{2(j-n)}
\left[ \binom{2j+1}{2n}iMf^{(2n)}+\binom{2j+1}{2n+1}f^{(2n+1)}\right].
\end{multline}
For convenience we omit the argument $x-c_0t$ of $f^{(k)}$. Using the easily checked identities $f'^2=f^2-f^4$ and $f''=f-2f^3$
we obtain for the higher derivatives of $f$
$$f^{(2n)}=\sum_{m=0}^n c_{nm}f^{2m+1} \quad \mbox{and} \quad f^{(2n+1)}=\sum_{m=0}^n c_{nm}(2m+1)f^{2m}f',$$
where we leave the coefficients $c_{nm}$ unspecified except for the fact that $c_{n0}=1$ for all $n\ge 0$. Inserting into
\eqref{c2} we obtain
\begin{equation}
 \label{c3}
\pd_x^{2j+1}u(x,t)=(-1)^{j+1}e^{i(Mx+\delta_0 t)}\sum_{m=0}^{j}f^{2m} \sum\nolimits_n,
\end{equation}
where
$$\sum\nolimits_n:=\sum_{n=m}^j(-1)^{n+1}c_{nm}M^{2(j-n)}\left[\binom{2j+1}{2n}iMf+\binom{2j+1}{2n+1}(2m+1)f'\right].$$
Similarly we see that
\begin{equation}
 \label{c4}
\pd_x^{2j}u(x,t)=e^{i(Mx+\delta_0 t)}\sum_{k=0}^{2j}\binom{2j}{k}(iM)^{2j-k}f^{(k)}
=(-1)^{j}e^{i(Mx+\delta_0 t)}\sum_{m=0}^{j}f^{2m} \sum\nolimits_n^{'}
\end{equation}
with
$$\sum\nolimits_n^{'}:=\sum_{n=m}^j(-1)^{n}c_{nm}M^{2(j-n)}\left[\binom{2j}{2n}f-\frac{i}{M}\binom{2j}{2n+1}(2m+1)f'\right].$$
Now $\delta_0$ and $c_0$ are chosen in such a way that the linear terms in \eqref{c1} and \eqref{c3} cancel, which gives
$$\pd_tu(x,t) + (-1)^{j+1}\pd_x^{2j+1}u(x,t)=e^{i(Mx+\delta_0 t)}\sum_{m=1}^{j}f^{2m} \sum\nolimits_n,$$
where the double sum on the right hand side is a linear combination of 
\begin{eqnarray*}
 F:=&\big(M^{2j-1}f^3,M^{2j-2}f^2f',\dots,Mf^3,f^2f';\\
& M^{2j-3}f^5,M^{2j-4}f^4f',\dots,Mf^5,f^4f';\\
& \vdots \\
& M^{3}f^{2j-1},M^{2}f^{2j-2}f',Mf^{2j-1},f^{2j-2}f';\\
& Mf^{2j+1},f^{2j}f'\big),
\end{eqnarray*}
which is ordered firstly by increasing degree and secondly (for fixed degree) by decreasing powers of $M$. Next we consider
the system
\begin{eqnarray*}
 U:= & \big(|u|^2\pd_x^{2j-1}u,(\pd_x|u|^2)\pd_x^{2j-2}u,\dots,(\pd_x^{2j-2}|u|^2)\pd_xu,(\pd_x^{2j-1}|u|^2)u;\\
& |u|^4\pd_x^{2j-3}u,(\pd_x|u|^4)\pd_x^{2j-4}u,\dots,(\pd_x^{2j-4}|u|^4)\pd_xu,(\pd_x^{2j-3}|u|^4)u;\\
& \vdots \\
& |u|^{2j-2}\pd_x^{3}u,(\pd_x|u|^{2j-2})\pd_x^{2}u,(\pd_x^{2}|u|^{2j-2})\pd_xu,(\pd_x^{3}|u|^{2j-2})u;\\
& |u|^{2j}\pd_xu, (\pd_x|u|^{2j})u\big),
\end{eqnarray*}
which is ordered in the same manner by degree and secondly by the decreasing number of derivatives on the single factor $u$. The length of both systems is $j(j+1)$. Concerning the terms $(\pd_x^{l}|u|^{2k})\pd_x^{2(j-k)+1-l}u$ with $1\le k\le j$,
$0 \le l \le 2(j-k)+1$, which appear in $U$, we observe first, that for even $l$
\begin{equation}
 \label{c5}
\pd_x^{l}|u|^{2k}= \sum_{\nu =0}^{\frac{l}{2}}b_{lk\nu}f^{2(k+\nu)}
\end{equation}
and for odd $l$
\begin{equation}
 \label{c6}
\pd_x^{l}|u|^{2k}= \sum_{\nu =0}^{\frac{l-1}{2}}\widetilde{b}_{lk\nu}f^{2(k+\nu)-1}f'
\end{equation}
with nonvanishing coefficients $b_{lk\nu}$ and $\widetilde{b}_{lk\nu}$. Combining \eqref{c3} and \eqref{c5} we see that for $l$
even $e^{-i(Mx+\delta_0 t)}(\pd_x^{l}|u|^{2k})\pd_x^{2(j-k)+1-l}u$ is a linear combination of members of the system $F$ with leading (in the sense of
the order of $F$) term $M^{2(j-k)+1-l}f^{2k+1}$. Essentially the same holds true for odd $l$, as we obtain from \eqref{c4} and
\eqref{c6}; in this case the leading term is $M^{2(j-k)+1-l}f^{2k}f'$. Because of $f'^2=f^2-f^4$ the product of the derivatives
in \eqref{c4} and \eqref{c6} produces terms of the same kind, but not of leading order, so that no cancellations occur. Now read
$F$ and $U$ als column vectors. Then the preceeding considerations show that there is a triangular $j(j+1) \times j(j+1)$-matrix $A$
with nonvanishing diagonal entries, such that $U=e^{i(Mx+\delta_0 t)}AF$. Since $A$ is invertible,
$e^{i(Mx+\delta_0 t)}\sum_{m=1}^{j}f^{2m} \sum\nolimits_n$ can be reexpressed in terms of $U$, which shows that $u$ in fact
solves an equation of type \eqref{complex}. Finally we have that $v_{N\om}(x,t)=\om u_{\frac{N}{\om}}(\om x, \om ^{2j+1}t)$, which
solves the same equation.
\end{proof}

Inserting the above family of solutions into the proof of \cite[Theorem 1.2]{KPV01} we obtain our ill-posedness result in
Proposition \ref{counter}. We will be brief, because only two modifications have to be made. Firstly, the exponents have to
be adjusted to the more general data spaces, as was already carried out for $j=1$ in \cite[Section 5]{G04}. Secondly, we have
to take care of the higher propagation speed $c \sim N^{2j}$ (instead of $N^2$) of the solutions of the higher order equations.
This leads to a faster separation of two highly concentrated solutions starting close together and hence to the increasing
(with $j$) lower threshold for uniformly continuous dependence.

\begin{proof}[Proof of Proposition \ref{counter}]
 We choose parameters $N \rightarrow \infty$, $\om = N^{-sr'}$, and $N_{1,2}\sim N$ with $|N_1-N_2|=\frac{C}{T}N^{sr'-(2j-1)}$,
where $C$ is a large constant and $T>0$ the assumed lifespan. Let $f(z) = {\rm{sech}}(z)$ as in the previous lemma. Then, due
to our choices and the assumption $s>-\frac{1}{r'}$, the function $\F_x f (\tfrac{\cdot - N}{\om})$ is concentrated around $N$,
so that for $v_{N\om}$ as in Lemma \ref{solution} we have
$$\|v_{N_k \om}(0)\|_{\widehat{H}^r_s(\R)} \ls 1,\qquad k=1,2.$$
A short computation along the lines of \cite[(2.10) on p. 626]{KPV01} gives
$$\|v_{N_1 \omega}(0) - v_{N_2 \omega}(0)\|_{\widehat{H}^r_s(\R)}\ls \omega ^{-\frac{1}{r}}N^s |N_1 -N_2| =
 \frac{C}{T} N^{2sr'-(2j-1)} \longrightarrow 0,$$
if $s<\frac{2j-1}{2r'}$, as assumed.
On the other hand, we have for any positive $T$
\begin{eqnarray*}
\|v_{N_1 \omega}(T) - v_{N_2 \omega}(T)\|_{\widehat{H}^r_s(\R)} & \gs &  N^s \|v_{N_1 \omega}(T) - v_{N_2 \omega}(T)\|_{\widehat{H}^r_0(\R)} \\ 
& \gs &  N^s \sup_{\|\phi\|_{L^r} \le 1}\langle v_{N_1 \omega}(T) - v_{N_2 \omega}(T), \F_x^{-1} \phi\rangle _{L^2}.
\end{eqnarray*}
Now the $v_{N_k \omega}(T), \,\,k=1,2$, are concentrated on intervals $I_k$ of size $\omega ^{-1}$ around
$$c_kT =  \sum_{n=0}^j (-1)^{n+1} \binom{2j+1}{2n+1}N_k^{2(j-n)}\om^{2n} T \sim N_k^{2j}T,$$
which are disjoint for $N^{2j-1}|N_1-N_2| T \gg \omega ^{-1} =N^{sr'}$, as is guaranteed by our choice of $N_1-N_2$. Taking
$$\F_x^{-1}\phi= c \omega ^{\frac{1}{r'}} \chi_{I_1} \frac{\overline{v_{N_1 \omega}(T)}}{|v_{N_1 \omega}(T)|},$$
where the factor $\omega ^{\frac{1}{r'}}$ ensures that $\|\phi\|_{L^r} \le 1$, we see that $\|v_{N_1 \omega}(T) - v_{N_2 \omega}(T)\|_{\widehat{H}^r_s(\R)}$ is bounded from below by
$$c N^s \omega ^{\frac{1}{r'}} \int |v_{N_1 \omega}(x,T)|dx = c \omega \int f(\omega x) dx = c .$$
Thus the mapping data upon solution from bounded subsets of $\widehat{H}^r_s(\R)$ to any solution space $X_T$ continuously embedded in $C([0,T],\widehat{H}^r_s(\R))$ cannot be uniformly continuous, if $-\frac{1}{r'}<s<\frac{2j-1}{2r'}$.
\end{proof}

\section{Estimates for the higher order KdV equations}
\label{moreestimates}

In this section we prove the estimates necessary for Theorem \ref{hoKdVlocal}. Here we rely heavily on the bilinear estimate
\begin{equation}
  \label{mixedbill}
\n{M_{p,j}(u_1, u_2)}{\widehat{L^r_x}(\widehat{L^p_t})} \ls \n{u_1}{X^{r,p}_{s_1,b}}\n{u_2}{X^{r,p}_{s_2,b}},
\end{equation}
which holds true for $1<r\le p \le 2$, $b>\frac{1}{p}$ and $s_{1,2}\ge 0$ with $s_1+s_2>\frac{1}{r}-\frac{1}{p}$. Inequality
\eqref{mixedbill} is easily obtained from \eqref{xBronstein} in Corollary \ref{xBill} by Sobolev type embeddings. In addition
we will need the following weaker version, which is usefull, if the frequencies of the two factors are very close together,
that is, if $|\xi_1-\xi_2| \ll |\xi_1| \sim |\xi_2|$. For its proof we go back to the arguments of Kenig, Ponce, and Vega in
\cite{KPV96a}, \cite{KPV96b}.

\begin{lemma}
 \label{weaksmooth}
Let $1<r<p\le 2$, $0<\e \ll 1$, and $b>\frac{1}{p}$. Then, with $P_{\ge1}=\F_x^{-1}\chi_{\{|\xi|\ge1\}}\F_x$,
\begin{equation}
 \label{weaksmoothest}
\n{D_x^{\frac{2j-1}{2r}-\e}P_{\ge1}(u_1 u_2)}{\widehat{L^r_x}(\widehat{L^p_t})} \ls \n{u_1}{X^{r,p}_{0,b}}\n{u_2}{X^{r,p}_{0,b}}.
\end{equation}
\end{lemma}

\begin{proof}
 For $i=1,2$ we choose $f_i$ with $\|f_i\|_{L^{r'}_{\xi}(L^{p'}_{\tau})}=\n{u_1}{X^{r,p}_{0,b}}$, so that \eqref{weaksmoothest}
can be written as
$$\||\xi|^{\frac{2j-1}{2r}-\e}\int_* d\xi_1d\tau_1\frac{f_1(\xi_1,\tau_1)f_2(\xi_2,\tau_2)}{\lb \sigma_1\rb^b\lb \sigma_2\rb^b}\|_{L^{r'}_{\xi}(L^{p'}_{\tau})}
\ls\|f_1\|_{L^{r'}_{\xi}(L^{p'}_{\tau})}\|f_2\|_{L^{r'}_{\xi}(L^{p'}_{\tau})}.$$
In view of the estimate \eqref{xBronstein} in Lemma \ref{xBill} it is sufficient to consider the frequency range
$|\xi_1-\xi_2| \ll |\xi_1|$, where $|\xi|$ is large. We decompose dyadically with respect to $|\xi_1-\tfrac{\xi}{2}|=\frac12 |\xi_1-\xi_2|$ and consider the contributions
$$\||\xi|^{\frac{2j-1}{2r}-\e}\int_* d\xi_1d\tau_1 \chi_{A_k}(\xi_1)\frac{f_1(\xi_1,\tau_1)f_2(\xi_2,\tau_2)}{\lb \sigma_1\rb^b\lb \sigma_2\rb^b}\|_{L^{r'}_{\xi}(L^{p'}_{\tau})},$$
where $A_k=\{\xi_1:|\xi_1-\tfrac{\xi}{2}|\sim 2^{-k}\}$ for $k\ge 1$, and $A_0=\{\xi_1:1\le |\xi_1-\tfrac{\xi}{2}|\ll |\xi_1|\}$.
We apply H\"older's inequality and \cite[Lemma 4.2]{GTV97} to obtain
$$\int d\tau_1\frac{f_1(\xi_1,\tau_1)f_2(\xi_2,\tau_2)}{\lb \sigma_1\rb^b\lb \sigma_2\rb^b}
\ls \lb\sigma_{res}\rb^{-b}\left(\int_* d\tau_1|f_1(\xi_1,\tau_1)f_2(\xi_2,\tau_2)|^{p'}\right)^{\frac{1}{p'}},$$
where $\sigma_{res}=\tau-\xi_1^{2j+1}-\xi_2^{2j+1}=\tau-\frac{\xi^{2j+1}}{2^{2j}}-h(x)$ with $x=\xi_1-\tfrac{\xi}{2}$ and
$$h(x)=\xi \sum_{l=1}^j\binom{ 2j+1 }{ 2l}(\tfrac{\xi}{2})^{2(j-l)}x^{2l}\sim \xi x^2 (\xi^{2(j-1)}+x^{2(j-1)}),$$
cf. \eqref{g1}. We have
\begin{equation}
 \label{h1}
h'(x)=\xi x\sum_{l=1}^j\binom{ 2j+1 }{ 2l}(\tfrac{\xi}{2})^{2(j-l)}2lx^{2(l-1)}\sim \xi x (\xi^{2(j-1)}+x^{2(j-1)}),
\end{equation}
which implies
\begin{equation}
 \label{h2}
|h'(x)| \gs (|h(x)||\xi|(\xi^{2(j-1)}+x^{2(j-1)}))^{\frac{1}{2}}\gs |h(x)|^{\frac12}|\xi|^{\frac{2j-1}{2}}.
\end{equation}
Now we treat first the case where $k=0$. A second H\"older application gives
\begin{eqnarray*}
 & \int_* d\xi_1d\tau_1 \chi_{A_0}(\xi_1)\frac{f_1(\xi_1,\tau_1)f_2(\xi_2,\tau_2)}{\lb \sigma_1\rb^b\lb \sigma_2\rb^b}\\
\ls & \int_* d\xi_1 \chi_{A_0}(\xi_1)|\xi_1-\tfrac{\xi}{2}|^{\frac{1}{p}}\lb\sigma_{res}\rb^{-b}
\lb\xi_1-\tfrac{\xi}{2}\rb^{-\frac{1}{p}}\left(\int_* d\tau_1|f_1(\xi_1,\tau_1)f_2(\xi_2,\tau_2)|^{p'}\right)^{\frac{1}{p'}}\\
\ls & \left( \int_{A_0} d\xi_1 |\xi_1-\tfrac{\xi}{2}| \lb\sigma_{res}\rb^{-bp}\right)^{\frac{1}{p}}
\left(\int_* d\xi_1d\tau_1\lb\xi_1-\tfrac{\xi}{2}\rb^{-\frac{p'}{p}}|f_1(\xi_1,\tau_1)f_2(\xi_2,\tau_2)|^{p'}\right)^{\frac{1}{p'}}.
\end{eqnarray*}
For the first factor we obtain the upper bound
$$\left(\int dh |\xi|^{-(2j-1)}\lb\tau-\frac{\xi^{2j+1}}{2^{2j}}-h\rb^{-bp}\right)^{\frac{1}{p}}
\ls |\xi|^{-\frac{2j-1}{p}}\ls |\xi|^{-\frac{2j-1}{2r}+\e},$$
so that the $|\xi|$-factors cancel. Using Fubini we arrive at
\begin{eqnarray*}
& \||\xi|^{\frac{2j-1}{2r}-\e}\int_* d\xi_1d\tau_1 \chi_{A_0}(\xi_1)\frac{f_1(\xi_1,\tau_1)f_2(\xi_2,\tau_2)}{\lb \sigma_1\rb^b\lb \sigma_2\rb^b}\|_{L^{p'}_{\tau}}\\
\ls & \left(\int_* d\xi_1 \lb\xi_1-\tfrac{\xi}{2}\rb^{-\frac{p'}{p}}
 \|f_1(\xi_1,\cdot)\|^{p'}_{L^{p'}_{\tau}}\|f_2(\xi_2,\cdot)\|^{p'}_{L^{p'}_{\tau}}\right)^{\frac{1}{p'}}\\
\ls & \left(\int_* d\xi_1 
 \|f_1(\xi_1,\cdot)\|^{r'}_{L^{p'}_{\tau}}\|f_2(\xi_2,\cdot)\|^{r'}_{L^{p'}_{\tau}}\right)^{\frac{1}{r'}}.
\end{eqnarray*}
Taking $\|\quad\|_{L^{r'}_{\xi}}$ of the latter, we obtain the desired bound for the contribution from $A_0$. For $k\ge 1$ the
argument is similar. Lemma 4.1 from \cite{GTV97} and H\"older's inequality give
\begin{multline*}
\int_* d\xi_1d\tau_1 \chi_{A_k}(\xi_1)\frac{f_1(\xi_1,\tau_1)f_2(\xi_2,\tau_2)}{\lb \sigma_1\rb^b\lb \sigma_2\rb^b}\\
\ls \left(\int_*d\xi_1 \chi_{A_k}(\xi_1)\lb\sigma_{res}\rb^{-bp}\right)^{\frac{1}{p}}\left(\int_*d\xi_1 d\tau_1\chi_{A_k}(\xi_1)|f_1(\xi_1,\tau_1)f_2(\xi_2,\tau_2)|^{p'}\right)^{\frac{1}{p'}}.
\end{multline*}
For the first factor we use \eqref{h1} and \eqref{h2} to get
\begin{multline*}
 \left( \int_{|x|\sim 2^{-k}}dx\lb\tau-\frac{\xi^{2j+1}}{2^{2j}}-h(x)\rb^{-bp}\right)^{\frac{1}{p}}
\ls\left( \int_{|x|\sim 2^{-k}}\frac{dh}{|h'(x)|}\lb\tau-\frac{\xi^{2j+1}}{2^{2j}}-h\rb^{-bp}\right)^{\frac{1}{p}}\\
\ls \left( \int_{|x|\sim 2^{-k}}dh |h|^{\frac{\theta -1}{2}}|x|^{-\theta}|\xi|^{-\frac{(2j-1)(1+\theta)}{2}}\lb\tau-\frac{\xi^{2j+1}}{2^{2j}}-h\rb^{-bp}\right)^{\frac{1}{p}}
\ls 2^{\frac{k\theta}{p}}|\xi|^{-\frac{(2j-1)(1+\theta)}{2p}},
\end{multline*}
whenever $\theta \in [0,1]$. We choose $\theta$ slightly smaller than $\frac{p-r}{r}$, so that $\frac{(2j-1)(1+\theta)}{2p}
= \frac{2j-1}{2r}-\e$ and the $|\xi|$-factors cancel again. Hence
\begin{eqnarray*}
& \||\xi|^{\frac{2j-1}{2r}-\e}\int_* d\xi_1d\tau_1 \chi_{A_k}(\xi_1)\frac{f_1(\xi_1,\tau_1)f_2(\xi_2,\tau_2)}{\lb \sigma_1\rb^b\lb \sigma_2\rb^b}\|_{L^{p'}_{\tau}}\\
\ls & 2^{\frac{k\theta}{p}}\left(\int_* d\xi_1 \chi_{A_k}(\xi_1)
 \|f_1(\xi_1,\cdot)\|^{p'}_{L^{p'}_{\tau}}\|f_2(\xi_2,\cdot)\|^{p'}_{L^{p'}_{\tau}}\right)^{\frac{1}{p'}}\\
\ls & 2^{k(\frac{\theta +1}{p}-\frac{1}{r})}\left(\int_* d\xi_1 
 \|f_1(\xi_1,\cdot)\|^{r'}_{L^{p'}_{\tau}}\|f_2(\xi_2,\cdot)\|^{r'}_{L^{p'}_{\tau}}\right)^{\frac{1}{r'}}
\end{eqnarray*}
where $\frac{\theta +1}{p}-\frac{1}{r}<0$. It remains to take the $L^{r'}_{\xi}$-norm and to sum up over $k$.
\end{proof}

Remark: Similar arguments show that \eqref{weaksmoothest} is also true for $p=r$ and $\e =0$. In this case the dyadic decomposition is not necessary.

\quad

We turn to the estimate for the quadratic nonlinearities. Here the resonance relation
\begin{equation}
 \label{rr2}
|\xi_1\xi_2(\xi_1+\xi_2)|(\xi_1^{2j-2}+\xi_2^{2j-2}) \ls \sum_{i=0}^2 \lb\sigma_i\rb
\end{equation}
is essential, to which Lemma \ref{resonance} reduces for $\xi_3=0$.

\begin{theorem}
 \label{quadratic}
Let $j\ge 2$, $l_0 \ge 1$, $l_{1,2}\ge0$ with $l_0+l_1+l_2=2j-1$, and $1<r\le p \le \frac{2j}{2j-1}$ as well as
$$s> \frac{2j+1}{p}-\frac{2j-1}{2r}-2= j - \frac12 - \frac{2j+1}{p'}+\frac{2j-1}{2r'}.$$
Then there exists $b'>-\frac{1}{p'}$, such that for all $b>\frac{1}{p}$
$$\|\pd_x^{l_0}(\pd_x^{l_1}u_1\pd_x^{l_2}u_2)\|_{X^{r,p}_{s,b'}}\ls \|u_1\|_{X^{r,p}_{s,b}}\|u_2\|_{X^{r,p}_{s,b}}.$$
\end{theorem}
The proof of this theorem consists again of a case by case discussion depending on the relative size of the frequencies
$\xi$, $\xi_1$, $\xi_2$ and the modulations $\sigma_0=\tau - \xi^{2j+1}$, $\sigma_1=\tau_1 - \xi_1^{2j+1}$, and
$\sigma_2=\tau_2 - \xi_2^{2j+1}$. It is similar to the proof of Theorem \ref{cubic} and we shall use partly the notation
introduced at the beginning of Section \ref{nonlin}\footnote{Besides that we sometimes write $a\pm $ for a number $a\pm \e$
with sufficiently small $\e >0$.}. The only new element is the extensive use of the resonance relation, which is standard in
the literature. So we will be brief and avoid repetitions as far as possible. The first part of the proof covers as well the case
$j=1$, while in the second part we restrict ourselves to $j\ge 2$.

\begin{proof}
 The trivial region, where $|\xi_1| \le 1$ and $|\xi_2| \le 1$ (and hence $|\xi| \le 2$) is easily treated for all $b' \le 0$
by Young's and H\"older's inequalities. The region with $|\xi_1| \le 1 \le |\xi_2|$ (or vice versa) is decisive. Here we apply
the estimate \eqref{mixedbill}, which requires especially $p \le \frac{2j}{2j-1}$ in order to control all the $2j-1$
derivatives in the nonlinearity. The estimate in this region is independent of $s \in \R$ and works for all $b' \le 0$. In the
sequel we will always have $|\xi_1| \ge 1$ and $|\xi_2| \ge 1$.

\quad

Case 1: $|\sigma_0 | \ge |\sigma_{1,2}|$. Here we may assume by symmetry that $|\xi_1| \ge |\xi_2|$.

\quad

Subcase 1.1: $|\xi_1| \sim |\xi_2|$.

\quad

Subsubcase 1.1.1: $|\xi| \ll |\xi_1|$. Here we apply the resonance relation \eqref{rr2}, which gives control over
$|\xi|^{-b'}|\xi_1|^{-2jb'}=|\xi|^{\frac{1}{p'}-}|\xi_1|^{\frac{2j}{p'}-}$, and the bilinear estimate \eqref{mixedbill},
which gives a further gain of $|\xi|^{\frac{1}{p}}|\xi_1|^{\frac{2j-1}{p}}$. By this and the $|\xi_1|^{2s}$ from the norms on the
right we have to control the $\lb \xi \rb^{s}$ from the norm on the left, the $|\xi||\xi_1|^{2j-2}$ from the nonlinearity and the
$|\xi_1|^{\frac{1}{r}-\frac{1}{p}+}$ from \eqref{mixedbill}. This works for $s=0$ (and hence for $s \ge 0$), since the 
$|\xi|$-factors almost cancel and $|\xi_1|^{2j-2+\frac{1}{r}-\frac{1}{p}} \ls |\xi_1|^{2j-1+\frac{1}{p'}-}$. If $s<0$, the
factor $\lb \xi \rb^{s}$ is useless, and so we need
\begin{equation}
 \label{cond1}
2j-2+\frac{1}{r}-\frac{1}{p}< \frac{2j-1}{p}+\frac{2j}{p'}+2s, \quad \mbox{i. e.} \quad s>-\frac12-\frac{1}{2r'},
\end{equation}
which is our first condition on $s$.

\quad

Subsubcase 1.1.2: $|\xi| \sim |\xi_1|$ (so that eventually $|\xi_1-\xi_2|\ll |\xi|$). Here we use Lemma \ref{weaksmooth} instead
of estimate \eqref{mixedbill}. Together with the resonance relation this gives a gain of $|\xi|^{\frac{2j-1}{2r}+\frac{2j+1}{p'}-}$.
Taking into account the $2j-1$ derivatives from the nonlinearity and the three $\lb \xi \rb^s$-factors from the norms, we are led
to the condition
\begin{equation}
 \label{cond2}
s> 2j-1-\frac{2j-1}{2r}-\frac{2j+1}{p'}= j-\frac12-\frac{2j+1}{p'}+\frac{2j-1}{2r'}.
\end{equation}
We observe that for $j \ge 2$ and in the admissible range of $p$ the condition \eqref{cond2} implies \eqref{cond1}.

\quad

Subcase 1.2: $|\xi_2| \ll |\xi_1|$ (so that $|\xi| \sim |\xi_1|$). Here we apply again the bilinear estimate \eqref{mixedbill},
which gives a $|\xi_1|^{\frac{2j}{p}}$, while from the resonance relation we get $|\xi_1|^{\frac{2j}{p'}-}|\xi_2|^{\frac{1}{p'}-}$.
Comparing the high frequencies first, we must have $\lb\xi\rb^s|\xi|^{2j-1} \ls |\xi_1|^{2j-}\lb\xi_1\rb^s$, which is obviously
fulfilled independently of $s$. Comparing all frequencies, taking also into account the loss of $|\xi_2|^{\frac{1}{r}-\frac{1}{p}+}$
in the application of \eqref{mixedbill}, we end up with the demand $s>-1-\frac{1}{r'}$, which is weaker that both conditions 
\eqref{cond1} and \eqref{cond2}.

\quad

Case 2: $|\sigma_1 | \ge |\sigma_{0,2}|$. Here we restrict ourselves to $j \ge 2$, so that the Sobolev regularity $s$ is always
positive.

\quad

Subcase 2.1: $|\xi_1|\sim |\xi_2|$. We apply the resonance relation \eqref{rr2} first, which gives the gain 
$|\xi|^{\frac{1}{p}}|\xi_1|^{\frac{2j}{p}}$. Throwing away the $\lb \sigma_0 \rb^{b'}$, we get for this subregion the upper bound
$$\|(D_x^s\Lambda^b u_1)(D_x^{\frac{2j+1}{p'}-2}u_2)\|_{\widehat{L^r_x}(\widehat{L^p_t})} \ls 
\|u_1\|_{X^{r,p}_{s,b}}\|D_x^{\frac{2j+1}{p'}-2}u_2\|_{\widehat{L^{\infty}_{xt}}} \ls \|u_1\|_{X^{r,p}_{s,b}}\|u_2\|_{X^{r,p}_{s,b}},$$
provided
\begin{equation}
 \label{cond3}
s>\frac{2j+1}{p'}+\frac{1}{r}-2=\frac{2j+1}{p'}-\frac{1}{r'}-1,
\end{equation}
which is weaker than \eqref{cond2}, since $j \ge 2$.

\quad

Subcase 2.2: $|\xi_2|\ll |\xi_1|$. Here the resonance relation \eqref{rr2} provides a gain of $|\xi_2|^{\frac{1}{p}}|\xi_1|^{\frac{2j}{p}}$. The argument used in the previous subcase applies here as well, since the small
factor $|\xi_2|^{\frac{1}{p}}$ can always be used for the Sobolev type embedding in the last step. So we end up again with
condition \eqref{cond3}.

\quad

Subcase 2.3: $|\xi_1|\ll |\xi_2|$. The gain from the resonance relation here is $|\xi_1|^{\frac{1}{p}}|\xi_2|^{\frac{2j}{p}}$,
giving the upper bound
\begin{multline*}
\|(D_x^{\frac{2j+1}{p'}-2}\Lambda^b u_1)(D_x^su_2)\|_{\widehat{L^r_x}(\widehat{L^p_t})} \ls \\
\|D_x^{\frac{2j+1}{p'}-2}\Lambda^bu_1\|_{\widehat{L^{\infty}_x}(\widehat{L^p_t})}\|D_x^su_2\|_{\widehat{L^r_x}(\widehat{L^\infty_t})} \ls \|u_1\|_{X^{r,p}_{s,b}}\|u_2\|_{X^{r,p}_{s,b}},
\end{multline*}
which requires \eqref{cond3} again.

\quad

By symmetry the discussion of the case $|\sigma_2 | \ge |\sigma_{0,1}|$ is unnecessary, and the proof is complete.
\end{proof}
Next we estimate the cubic and higher nonlinearities in the KdV hierarchy. We do not aim for optimality here, since the most restrictive condition comes anyway from the quadratic terms.

\begin{lemma}
 \label{cubic+}
Let $j \ge 2$, $3 \le k \le j+1$, $l_0 \ge 1$, $l_1,\dots,l_k \ge 0$ with $\sum_{i=0}^k l_i = 2(j-k)+3$ and
$1<r\le p \le \frac{2j}{2j-1}$. Then, for $b>\frac{1}{p}$ and $s>\frac{2(j-k)+3}{2(k-1)}(1+\frac{1}{r'})$ the estimate
$$\|J^s \pd_x^{l_0}\prod_{i=1}^k \pd_x^{l_i}u_i\|_{\widehat{L^r_x}(\widehat{L^p_t})} \ls \prod_{i=1}^k\|u_i\|_{X^{r,p}_{s,b}}$$
holds true.
\end{lemma}

\begin{proof}
 By symmetry we may assume $|\xi_1|\ge \dots\ge|\xi_k|$. Then we have to distinguish only two cases.

\quad

Case 1: $|\xi_1|\ge 2 |\xi_k|$. Our assumptions imply $\frac{2j+1}{p'}<(k-1)(2-\frac{1}{r})$ and hence
$2(j-k)+3 +\frac{k-1}{r}-\frac{1}{p} < \frac{2j}{p}$. This gives
$$J^s \pd_x^{l_0}\prod_{i=1}^k \pd_x^{l_i}u_i \preceq M_{p,j}(J^su_1,J^{\frac{1}{p}-\frac{1}{r}-}u_k)
\prod_{i=2}^{k-1}J^{-\frac{1}{r}-}u_i.$$
Now the bilinear estimate \eqref{mixedbill} is combined with Sobolev type embeddings, which gives the desired bound in this case
for all $s \ge 0$.

\quad

Case 2: $|\xi_1|\sim |\xi_k|$. Certainly two frequencies have the same sign. To the corresponding factors we will apply Lemma \ref{weaksmooth}. We have to control $2(j-k)+3$ derivatives from the nonlinearity, $s$ derivatives from the norm, and $\frac{k-2}{r}+$
derivatives, which we have to spend on embeddings. On the other hand we have a gain of $\frac{2j-1}{2r}-$ derivatives from
\eqref{weaksmoothest} and $ks$ derivatives from the norms on the right. This leads to the condition
$$2(j-k)+3 +s+\frac{k-2}{r}< \frac{2j-1}{2r}+ks,$$
which is $s>\frac{2(j-k)+3}{2(k-1)}(1+\frac{1}{r'})$, as assumed.
\end{proof}

We observe that the lower bound on $s$ in Lemma \ref{cubic+} is smaller than that in Theorem \ref{quadratic}, which is decreasing
in $p$. So we choose $p=\frac{2j}{2j-1}$, and the decisive lower bound on $s$ becomes
$$s>j-\frac32-\frac{1}{2j}+\frac{2j-1}{2r'},$$
just as demanded in Theorem \ref{hoKdVlocal}.

\section{Further estimates related to the KdV equation}
 \label{lastsection}

In this last section we refine the analysis for the quadratic terms in the special case $j=1$. Notice that the $X_{s,b}^{r,p}$-norms in the sequel are always those with phase function $\phi(\xi)=\xi^3$. In the proof of Theorem \ref{quadratic} the region with
$|\sigma_{1}|\ge|\sigma_{0,2}|$ was estimated roughly by using nothing but the resonance relation and Sobolev type embeddings.
This was sufficient for $j\ge 2$, since another region gave a stronger condition on $s$, but has to be sharpened in order to obtain
Proposition \ref{KdVlocal}. To exploit the $\lb\sigma_0\rb$-weight we will apply the following lemma, which exhibits at least a weak
smoothing effect.
\begin{lemma}
 \label{lastbutone}
Let $1<r \le p \le 2$, $b' < - \frac{1}{p'}$, $b>\frac{1}{p}$, and $s_{1,2}\ge 0$. Then,
\begin{itemize}
 \item[a)] if $s_1+s_2>\frac{1}{r}-\frac{1}{r'}$, we have
\begin{equation}
 \label{mixedbill*}
\|M^*_{r',1}(u_1,u_2)\|_{X^{r,p}_{0,b'}}\ls \|J^{s_1}u_1\|_{\widehat{L^r_x}(\widehat{L^p_t})}\|u_2\|_{X^{r,p}_{s_2,b}},
\end{equation}
 \item[b)] if $\F_xu_1$ is supported outside $[-1,1]$ and $s_1+s_2>1-\frac{7}{2r'}$, the estimate
\begin{equation}
 \label{weaksmoothest*}
\|(D_x^{\frac{1}{2r'}}u_1)u_2\|_{X^{r,p}_{0,b'}}\ls \|J^{s_1}u_1\|_{\widehat{L^r_x}(\widehat{L^p_t})}\|u_2\|_{X^{r,p}_{s_2,b}}
\end{equation}
holds true.
\end{itemize}
\end{lemma}
\begin{proof}
 By a Sobolev type embedding in time, it is sufficient in both cases to prove the upper bound
$$\dots \ls \|J^{s_1}u_1\|_{\widehat{L^r_x}(\widehat{L^p_t})}\|u_2\|_{X^{r,p'}_{s_2,-b'}}.$$
For that purpose we choose $f_{1,2}$ with $\|f_1\|_{L^{r'}_{\xi}(L^{p'}_{\tau})}=\|J^{s_1}u_1\|_{\widehat{L^r_x}(\widehat{L^p_t})}$
and $\|f_2\|_{L^{r'}_{\xi}(L^{p}_{\tau})}=\|u_2\|_{X^{r,p'}_{s_2,-b'}}$. Then we have to show
\begin{multline*}
 \|\lb\sigma_0\rb^{b'}\int_* d\xi_1d\tau_1 W(\xi,\xi_1)\lb\xi_1\rb^{-s_1}\lb\xi_2\rb^{-s_2}
\lb\sigma_2\rb^{b'}f_1(\xi_1,\tau_1)f_2(\xi_2,\tau_2)\|_{L^{r'}_{\xi}(L^{p'}_{\tau})}\\
\ls\|f_1\|_{L^{r'}_{\xi}(L^{p'}_{\tau})}\|f_2\|_{L^{r'}_{\xi}(L^{p}_{\tau})},\hspace{4cm}
\end{multline*}
where the weight $W(\xi,\xi_1)$ will be specified later. Testing with a generic function $\phi \in L^{r}_{\xi}(L^{p}_{\tau})$,
using Fubini's theorem and H\"older's inequality we reduce matters to showing that
\begin{multline}
 \label{x}
\|\lb\xi_1\rb^{-s_1}\int_* d \xi d\tau W(\xi,\xi_1)\lb\xi_2\rb^{-s_2}
\lb\sigma_0\rb^{b'}\lb\sigma_2\rb^{b'}\phi(\xi,\tau)f_2(\xi_2,\tau_2)\|_{L^{r}_{\xi_1}(L^{p}_{\tau_1})}\\
\ls \|\phi\|_{L^{r}_{\xi}(L^{p}_{\tau})}\|f_2\|_{L^{r'}_{\xi}(L^{p}_{\tau})}.\hspace{4cm}
\end{multline}
By H\"older's inequality and \cite[Lemma 4.2]{GTV97} we have
$$\int_* d \tau \lb\sigma_0\rb^{b'}\lb\sigma_2\rb^{b'}\phi(\xi,\tau)f_2(\xi_2,\tau_2)
\ls\lb\sigma_{res}\rb^{b'}\left(\int_* d\tau|\phi(\xi,\tau)f_2(\xi_2,\tau_2)|^p\right)^{\frac{1}{p}},$$
where $\sigma_{res}=\tau_1-\xi_1^3+3\xi_1\xi(\xi-\xi_1)=\tau_1-\frac{\xi_1^3}{4}+3\xi_1x^2$ with $x=\xi-\tfrac{\xi_1}{2}$.
Writing $W(\xi,\xi_1)=W_1(\xi,\xi_1)W_2(\xi,\xi_1)$ we obtain by a second H\"older application
\begin{eqnarray*}
 \int_* d \xi d\tau W(\xi,\xi_1)\lb\xi_2\rb^{-s_2}
\lb\sigma_0\rb^{b'}\lb\sigma_2\rb^{b'}\phi(\xi,\tau)f_2(\xi_2,\tau_2)\hspace{0.8cm}\\
\ls \Big( \int_* d \xi W_1(\xi,\xi_1)^{r'} \lb\sigma_{res}\rb^{b'r'}\Big)^{\frac{1}{r'}} \times \hspace{3cm}\\
\Big(\int_* d \xi W_2(\xi,\xi_1)^{r}\lb\xi_2\rb^{-s_2r} \Big(\int_*d \tau|\phi(\xi,\tau)f_2(\xi_2,\tau_2)|^p\Big)^{\frac{r}{p}}\Big)^{\frac{1}{r}}
\end{eqnarray*}
Assume now the first factor to be bounded. Then we can use Minkowski's inequality to estimate the left hand side of \eqref{x} by
\begin{multline}
 \label{xx}
\|\lb\xi_1\rb^{-s_1} \Big(\int d \xi W_2(\xi,\xi_1)^{r}\lb\xi_2\rb^{-s_2r}
\|\phi(\xi,\cdot)\|^r_{L^{p}_{\tau}}\|f_2(\xi_2,\cdot)\|^r_{L^{p}_{\tau}} \Big)^{\frac{1}{r}}\|_{L^{r}_{\xi_1}}\\
= \|\|\phi(\xi,\cdot)\|_{L^{p}_{\tau}}\Big(\int d\xi_1 \lb\xi_1\rb^{-s_1r}\lb\xi_2\rb^{-s_2r}
W_2(\xi,\xi_1)^{r}\|f_2(\xi_2,\cdot)\|^r_{L^{p}_{\tau}}\Big)^{\frac{1}{r}}\|_{L^{r}_{\xi}}.
\end{multline}
Now we specify the weights. For part a) we choose $W_1(\xi,\xi_1)=(|\xi_1||\xi-\tfrac{\xi_1}{2}|)^{\frac{1}{r'}}$ and $W_2(\xi,\xi_1)=1$, so that with $h(\xi)=\xi_1(\xi-\tfrac{\xi_1}{2})^2$
$$\int d \xi W_1(\xi,\xi_1)^{r'} \lb\sigma_{res}\rb^{b'r'} = \int dh \lb\tau_1 - \tfrac{\xi_1^3}{4}+3h\rb^{b'r'}\le c.$$
The right side of \eqref{xx} is then estimated by H\"older's inequality using $s_1+s_2>\frac{1}{r}-\frac{1}{r'}$, which gives
\eqref{x}. For part b) we distinguish two cases. If $|\xi - \tfrac{\xi_1}{2}|\le 1$ we choose
$W_1(\xi,\xi_1)=|\xi_1|^{\frac{1}{2r'}}$ and $W_2(\xi,\xi_1)=\chi_{\{|\xi - \tfrac{\xi_1}{2}|\le 1\}}$. Then
\begin{multline*}
\int d \xi W_1(\xi,\xi_1)^{r'} \lb\sigma_{res}\rb^{b'r'} 
=\int dx |\xi_1|^{\frac12}\lb\tau_1 - \tfrac{\xi_1^3}{4}+3\xi_1x^2\rb^{b'r'}\\
= \frac12 \int dh |h|^{-\frac12} \lb\tau_1 - \tfrac{\xi_1^3}{4}+3h\rb^{b'r'}\le c.
\end{multline*}
The right hand side of \eqref{xx} is again estimated by H\"older's inequality, which works for all nonnegative $s_{1,2}$.
If $|\xi - \tfrac{\xi_1}{2}|\ge 1$ we choose $W_1(\xi,\xi_1)=(|\xi_1||\xi - \tfrac{\xi_1}{2}|)^{\frac{1}{r'}}$ and 
$W_2(\xi,\xi_1)=|\xi_1|^{-\frac{1}{2r'}}|\xi - \tfrac{\xi_1}{2}|^{-\frac{1}{r'}}\chi_{\{|\xi - \tfrac{\xi_1}{2}|\ge 1\}}$,
so that $W(\xi,\xi_1)=|\xi_1|^{\frac{1}{2r'}}\chi_{\{|\xi - \tfrac{\xi_1}{2}|\ge 1\}}$. We have already seen that in this case
$\int d \xi W_1(\xi,\xi_1)^{r'} \lb\sigma_{res}\rb^{b'r'} \le c$, and the right hand side of \eqref{xx} becomes bounded by
\begin{multline*}\|\|\phi(\xi,\cdot)\|_{L^{p}_{\tau}}\|\lb\xi_1\rb^{-s_1-\frac{1}{2r'}}\lb\xi_2\rb^{-s_2}\lb\xi - \tfrac{\xi_1}{2}\rb^{-\frac{1}{r'}}\|f_2(\xi_2,\cdot)\|_{L^{p}_{\tau}}\|_{L^{r}_{\xi_1}}\|_{L^{r}_{\xi}}\\
\ls \|\phi\|_{L^{r}_{\xi}(L^{p}_{\tau})}\|f_2\|_{L^{r'}_{\xi}(L^{p}_{\tau})},
\end{multline*}
provided $s_1+s_2+\frac{3}{2r'}>\frac{1}{r}-\frac{1}{r'}$, as assumed.
\end{proof}
Equipped with Lemma \ref{lastbutone} we can turn to the estimate crucial for Proposition \ref{KdVlocal}.
\begin{lemma}
 \label{last}
Assume one of the following conditions to be fulfilled.
\begin{itemize}
 \item[a)] $1<r\le \frac75$, $s>-\frac14-\frac{11}{8r'}$, and $\frac{1}{p'}=\frac14 + \frac{5}{8r'}$;
 \item[b)] $\frac75 \le r \le 2$, $s>-\frac12-\frac{1}{2r'}$, and $\frac13(1+\frac{1}{r'})\le
 \frac{1}{p'}\le\min{(\frac12,\frac{3}{2r'})}$.
\end{itemize}
Then there exists $b'>-\frac{1}{p'}$, so that for all $b>\frac{1}{p}$ the estimate
$$\|\pd_x(u_1u_2)\|_{X^{r,p}_{s,b'}}\ls\|u_1\|_{X^{r,p}_{s,b}}\|u_2\|_{X^{r,p}_{s,b}}$$
holds true.
\end{lemma}
\begin{proof}
 We modify the analysis in the proof of Theorem \ref{quadratic}. The cases $|\xi_1|\le1$ and/or $|\xi_2|\le1$ as well as
$|\sigma_0|\ge|\sigma_{1,2}|$ remain unchanged. These led to the conditions $1<r \le p\le2$, $s>-\frac12-\frac{1}{2r'}$
\eqref{cond1}, and $s>\frac12-\frac{3}{p'}+\frac{1}{2r'}$ \eqref{cond2}\footnote{specialized to $j=1$}, which are easily checked.
By symmetry it remains only to discuss the region $|\sigma_1|\ge|\sigma_{0,2}|$. We distinguish several subcases.

\quad

Subcase 1: $|\xi_1|\ll|\xi_2|$. We apply part a) of Lemma \ref{lastbutone}, which gives a gain of
$|\xi_1|^{\frac{1}{r'}}|\xi_2|^{\frac{1}{r'}}$ and a loss of $|\xi_1|^{\frac{1}{r}-\frac{1}{r'}}$. We have to control
$|\xi|^{1+s}$ (from the nonlinearity and the norm on the left) and $|\xi_1|^{-s}|\xi_2|^{-s}$ (from the norms on the right),
whereas the resonance relation gives a gain of $|\xi_1|^{\frac{1}{p}-}|\xi_2|^{\frac{2}{p}-}$. Expressing all losses and gains in terms of $\xi_1$ and $\xi_2$, we must have
$$|\xi_1|^{\frac{1}{r}-\frac{1}{r'}-s}|\xi_2| \ls |\xi_1|^{\frac{1}{r'}+\frac{1}{p}-}|\xi_2|^{\frac{2}{p}+\frac{1}{r'}-},$$
which is fulfilled in its high frequency part, since $\frac{2}{p}\ge 1$ and $\frac{1}{r'}>0$. Comparing both, the high and the low
frequencies, we find the condition
\begin{equation}
 \label{cond13}
s>\frac{3}{p'}-\frac{4}{r'}-1.
\end{equation}

\quad

Subcase 2: $|\xi_2|\ll|\xi_1|$. The estimation here follows the same lines as in subcase 1 and leads again to condition \eqref{cond13}. It is in fact the more harmless case, since the gain in the application of Lemma \ref{lastbutone} lies completely
on the high frequency.

\quad

Subcase 3: $|\xi_1|\sim|\xi_2|\gg|\xi|$. Again we can apply part a) of Lemma \ref{lastbutone}, where now the loss falls unavoidably
onto a high frequency. Comparing gains and losses we must have
$$\lb\xi\rb^s|\xi||\xi_1|^{\frac{1}{r}-\frac{1}{r'}-2s+}\ls |\xi|^{\frac{1}{p}-}|\xi_1|^{\frac{2}{r'}+\frac{2}{p}-}.$$
Considering all frequencies here, we end up with \eqref{cond13} again, but for the high frequency we need
\begin{equation}
 \label{cond14}
s>\frac{1}{p'}-\frac{2}{r'}-\frac12,
\end{equation}
which is more restrictive than \eqref{cond13}, since our assumptions imply $\frac{1}{p'}\le \frac14 + \frac{1}{r'}$.

\quad

Subcase 4: $|\xi_1|\sim|\xi_2|\sim|\xi|$. Here the symbol of the Fourier multiplier $M^*_{r',1}$ can become small, so we shall use part b) of Lemma \ref{lastbutone}. As long as $\frac{1}{r'}>\frac27$ there are no expenses in its application, but a gain of
$|\xi|^{\frac{1}{2r'}}$. Comparing all gains and losses we are led to the condition $s>\frac{3}{p'}-\frac{1}{2r'}-2$, which is
weaker than \eqref{cond1}, since $\frac{1}{p'}\le \frac12$. If $\frac{1}{r'} \le \frac27$, we must have $s>\frac{3}{p'}-\frac{4}{r'}-1$, which is \eqref{cond13} again.

\quad

Here the case by case discussion is completed, and it remains to check \eqref{cond14}. For $1<r\le \frac75$ the choice
$\frac{1}{p'}=\frac14 + \frac{5}{8r'}$ implies that both conditions \eqref{cond2} and \eqref{cond14} become $s>-\frac14-\frac{11}{8r'}$, which is assumed and stronger than \eqref{cond1}. For $\frac75 \le r \le 2$ condition \eqref{cond1}
is assumed and - in the allowed range for $\frac{1}{p'}$ - stronger than \eqref{cond2} and \eqref{cond14}.
\end{proof}

\end{document}